\documentclass[11pt]{article}

\usepackage[margin=1in]{geometry}
 
\usepackage[]{amsmath,amssymb,epsfig}
\usepackage{amsthm}
\usepackage{amsmath}
\usepackage{amssymb}
\usepackage{graphicx}
\usepackage{epstopdf}
\usepackage{comment}
\usepackage{array}
\usepackage{algorithm}
\usepackage{url}
\usepackage{enumitem}

\usepackage{lscape}
\usepackage{algpseudocode}
\usepackage{setspace}
\usepackage{multicol}
\usepackage{multirow}
\usepackage{color}
\usepackage{colortbl}
\usepackage{xcolor}

\usepackage{placeins}

\usepackage[%
    font={small,sf},
    labelfont=bf,
    format=hang,    
    format=plain,
    margin=0pt,
    width=0.8\textwidth,
]{caption}
\usepackage[list=true]{subcaption}

\newtheorem{theorem}{Theorem}

\newtheorem{corollary}{Corollary}

\newtheorem{lemma}{Lemma}

\newtheorem{proposition}{Proposition}
\newtheorem{remark}{Remark}

\numberwithin{equation}{section}

\DeclareMathOperator{\diag}{diag}

\newcommand{\rev}[1]{\textcolor{black}{ #1}}

\newcommand{\secrev}[1]{\textcolor{black}{ #1}}

\newcommand{\calC}{\ensuremath{\mathcal{C}}}

\newcommand{\calH}{\ensuremath{\mathcal{H}}}

\newcommand{\calL}{\ensuremath{\mathcal{L}}}

\newcommand{\calN}{\ensuremath{\mathcal{N}}}

\newcommand{\norm}[1]{\left\|{#1}\right\|}
\newcommand{\abs}[1]{\left|{#1}\right|}

\newcommand{\set}[1]{\left\{{#1}\right\}}
\newcommand{\dotprod}[2]{\langle#1,#2\rangle}
\newcommand{\expn}[1]{e^{{-{#1}\pi^2\sigma^2}}}
\newcommand{\expp}[1]{e^{{{#1}\pi^2\sigma^2}}}
\newcommand{\est}[1]{\widehat{#1}}
\newcommand{\expec}{\ensuremath{\mathbb{E}}}
\newcommand{\matR}{\ensuremath{\mathbb{R}}}
\newcommand{\matZ}{\ensuremath{\mathbb{Z}}}

\newcommand{\prob}{\ensuremath{\mathbb{P}}}

\newcommand{\real}{\text{Re}}
\newcommand{\imag}{\text{Im}}
\newcommand{\gest}{\ensuremath{\est{g}}}
\newcommand{\gbar}{\ensuremath{\overline{g}}}
\newcommand{\gbarest}{\ensuremath{\bar{\est{g}}}}
\newcommand{\zbar}{\ensuremath{\overline{z}}}

\newcommand{\smooth}{\ensuremath{B}}
\newcommand{\reg}{\ensuremath{\gamma}}
\newcommand{\lambtil}{\ensuremath{\widetilde{\lambda}}}
\newcommand{\lambbar}{\ensuremath{\overline{\lambda}}}
\newcommand{\lambmin}{\ensuremath{\lambda_{\min}}}

 \newcommand{\vecxi}{\mathbf \xi}  

\newcommand{\mproj}[1]{\frac{#1}{\abs{#1}}}

\title{Error analysis for denoising smooth modulo signals on a graph}

\author{Hemant Tyagi\thanks{Inria, Univ. Lille, CNRS, UMR 8524 - Laboratoire Paul Painlev\'{e}, F-59000} \\
\texttt{hemant.tyagi@inria.fr}}

\begin{document}
\maketitle

\begin{abstract}
In many applications, we are given access to noisy \emph{modulo} samples of a smooth function with the goal being to robustly unwrap the samples, i.e., to estimate the original samples of the function. In a recent work, Cucuringu and Tyagi \cite{CMT18_long} proposed denoising the modulo samples by first representing them on the unit complex circle and then solving a smoothness regularized least squares problem -- the smoothness measured w.r.t the Laplacian of a suitable proximity graph $G$ -- on the product manifold of unit circles. This problem is a quadratically constrained quadratic program (QCQP) which is nonconvex, hence they proposed solving its \emph{sphere-relaxation} leading to a trust region subproblem (TRS). In terms of theoretical guarantees, $\ell_2$ error bounds were derived for (TRS). These bounds are however weak in general and do not really demonstrate the denoising performed by (TRS). 

In this work, we analyse the (TRS) as well as an unconstrained relaxation of (QCQP). For both these estimators we provide a refined analysis in the setting of Gaussian noise and derive noise regimes where they provably denoise the modulo observations w.r.t the $\ell_2$ norm. The analysis is performed in a general setting where $G$ is any connected graph.
\end{abstract}

%
\section{Introduction} \label{sec:intro}
Many modern applications involve the acquisition of noisy \emph{modulo} samples of a function $f: \matR^d \rightarrow \matR$, i.e., we obtain 
\begin{equation} \label{eq:noisy_mod}
y_i = (f(x_i) + \eta_i)\bmod \zeta; \quad i=1,\dots,n 
\end{equation}
for some $\zeta \in \matR^+$, where $\eta_i$ denotes noise. Here, $a\bmod \zeta$ lies in the interval $[0,\zeta)$ and is such that $a = q \zeta + (a\bmod \zeta)$ for an integer $q$. This situation arises, for instance, in self-reset analog to digital converters (ADCs) which handle voltage surges by simply storing the modulo value of the voltage signal \cite{Kester,RHEE03,yamaguchi16}. In other words, if the voltage signal exceeds the range $[0,\zeta]$, then its value is simply reset via the modulo operation. Another important application is phase unwrapping where one typically obtains noisy modulo $2\pi$ samples. There, one usually seeks to infer the structure of an object by transmitting waveforms, and measuring the difference in phase (in radians) between the transmitted and scattered waveforms. This is common in synthetic radar aperture interferometry (InSAR) where one aims to learn the elevation map of a terrain (see for e.g. \cite{graham_insar,zebker86}), and also arises in many other domains such as MRI \cite{hedley92,laut73} and diffraction tomography \cite{pratt88}, to name a few. 

Let us assume $\zeta = 1$ from now. Given \eqref{eq:noisy_mod}, one can ask whether we can recover the original samples of $f$, i.e., $f(x_i)$? Clearly, this is only possible up to a global integer shift. Furthermore, answering this question requires making additional assumptions about $f$, for instance, we may assume $f$ to be smooth (for e.g., Lipschitz, continuously differentiable etc.). In this setting, when $d=1$, it is not difficult to see that we can \emph{exactly} recover the samples of $f$ when there is no noise (i.e., $\eta_i = 0$ in \eqref{eq:noisy_mod}) provided the sampling is fine enough. This is achieved by sequentially traversing the samples and reconstructing the quotient $q_i \in \matZ$ corresponding to $f(x_i)$ by setting it to $q_{i-1} + a_{i-1,i}$ where $a_{i-1,i}$ equals either (a) $0$ (if $y_i$ is ``close enough'' to $y_{i-1}$), or (b) $\pm 1$ (if $y_i$ is sufficiently smaller/larger than $y_{i-1}$). If $f$ is smooth, then a fine enough sampling density will ensure that nearby samples of $f$ have quotients which differ by either $1,-1$ or $0$. This argument can be extended to the general multivariate setting as well (see \cite{fanuel20}). 

While exact recovery of $f(x_i)$ is of course no longer possible in the presence of noise,  one can show \cite{fanuel20} that if $\eta_i\bmod 1$ is not too large (uniformly for all $i$), and if $n$ is sufficiently large, then the estimates $\tilde{f}(x_i)$ returned by the above procedure are such that (up to a global integer shift) for each $i$, $\abs{\tilde{f}(x_i) - f(x_i)}$ is bounded by a term proportional to the noise level. This suggests a natural two stage procedure for estimating $f(x_i)$ -- first obtain denoised estimates of $f(x_i) \mod 1$, and in the second stage, ``unwrap'' these denoised modulo samples to obtain the estimates $\tilde{f}(x_i)$. 

This was the motivation behind the recent work of Cucuringu and Tyagi \cite{cmt_aistats18,CMT18_long} that focused primarily on the first (modulo denoising) stage, which is an interesting question by itself. Before discussing their result, it will be convenient to fix the notation used throughout the paper.
\paragraph{Notation.} Vectors and matrices are denoted by lower and upper case symbols respectively, while sets are denoted by calligraphic symbols (e.g., $\calN$), 
with the exception of $[n] = \set{1,\dots,n}$ for $n \in \mathbb{N}$. The imaginary unit is denoted by $\iota = \sqrt{-1}$. For $a,b \geq 0$, 
we write $a \lesssim b$ when there exists an absolute constant $C > 0$ such that $a \leq C b$. If $a \lesssim b$ and $a \gtrsim b$, then we write $a \asymp b$. For $u \in \mathbb{C}^n$, we write $u = \real(u) + \iota \imag(u)$ where $\real(u), \imag(u) \in \matR^n$ denote its real and imaginary parts respectively. The Hermitian conjugate of $u$ is denoted by $u^*$, and  $\norm{u}_p$ denotes the usual $\ell_p$ norm in $\mathbb{C}^n$ for $1 \leq p \leq \infty$. 
\subsection{Denoising smooth modulo signals} \label{subsec:den_smooth_mod_intro}
Cucuringu and Tyagi \cite{CMT18_long, cmt_aistats18} considered the problem of denoising modulo samples by formulating it as an optimization problem on a manifold. Assume $f:[0,1]^d \rightarrow \matR$, and suppose that $x_i$'s form a uniform grid in $[0,1]^d$. The main idea is to represent each sample $y_i$ via $z_i = \exp(\iota 2\pi y_i)$ and observing that if $x_i$ is close to $x_j$, then $\exp(\iota 2\pi f(x_i)) \approx \exp(\iota 2\pi f(x_j))$ holds provided $f$ is smooth. In other words, the samples $z = (z_1,\dots,z_n)$ lie on the manifold 
$$\calC_n := \set{u \in \mathbb{C}^n: \abs{u_i} = 1; \ i=1,\dots,n}$$
which is the product manifold of unit radius circles, 
i.e., $\calC_n = \calC_1 \times \cdots \times \calC_1.$ Hence, by constructing a proximity graph $G = ([n], E)$ on the sampling points $(x_i)_{i=1}^n$ -- where there is an edge between $i$ and $j$ if $x_i$ is within a specified distance of $x_j$ -- they proposed solving  a quadratically constrained quadratic program (QCQP)
\begin{equation} 
\min_{g \in \calC_n} \norm{g - z}_2^2 + \reg g^* L g \label{prog:qcqp} \tag{$\text{QCQP}$}
\end{equation}
where $L$ is the Laplacian of $G$, and $\reg > 0$ is a regularization parameter. While the objective function is convex, this is a non-convex problem due to the constraint set $\calC_n$. It is not clear whether the global solution of \eqref{prog:qcqp} can be obtained in polynomial time. Hence they proposed relaxing  $\calC_n$ to a sphere leading to the trust region subproblem (TRS) 
\begin{equation} 
\min_{g \in \mathbb{C}^n: \norm{g}_2^2 = n} \norm{g - z}_2^2 + \reg g^* L g  \equiv \min_{g \in \mathbb{C}^n: \norm{g}_2^2 = n} -2\real(g^* z) + \reg g^* L g .
\label{prog:trs} \tag{$\text{TRS}$}
\end{equation}

Trust region subproblems are well known to be solvable in polynomial time, with many efficient solvers (e.g., \cite{Hager01,Naka17}). 
In \cite{CMT18_long,cmt_aistats18}, \eqref{prog:trs} was shown to be quite robust to noise and scalable to large problem sizes through extensive experiments. On the theoretical front, a $\ell_2$ stability result was shown for the solution $\gest$ of (TRS). To describe this result, denote $h \in \calC_n$ to be the ground truth samples where $h_i = \exp(\iota 2\pi f(x_i))$, and define the entry-wise projection operator $\mathbb{C}^n \mapsto \calC_n$ \cite{Liu2017}
\begin{equation} \label{eq:project_manifold}
    \left( \mproj{g} \right)_i :=  \left\{
\begin{array}{rl}
\frac{g_i}{\abs{g_i}} \ ; & g_i \neq 0, \\
1 \ ; & \text{ otherwise, }
\end{array} \right. \quad i=1,\dots,n.
\end{equation}
Consider the univariate case with $f:[0,1] \rightarrow \matR$, and assume that the (Gaussian) noise $\eta_i \sim \calN(0,\sigma^2)$ i.i.d for $i=1,\dots,n$. 
Assuming $f$ is $M$-Lipschitz, and $\sigma \lesssim 1$, denote $\Delta$ to be the maximum degree of $G$ with $x_i$'s forming a uniform grid on $[0,1]$. Then provided $\reg \Delta \lesssim 1$, the bounds in \cite[Theorem 14]{CMT18_long} imply\footnote{The result in \cite{CMT18_long} bounds $\norm{\gest - h}_2$ but we can then use $\norm{\mproj{\gest} - h}_2 \leq 2 \norm{\gest - h}_2$, see \cite[Proposition 3.3]{Liu2017}. Also, while the statement in \cite[Theorem 14]{CMT18_long} has a more complicated form than that stated in \eqref{eq:cmt_err_bd}, one can verify that it is essentially of the same order as in \eqref{eq:cmt_err_bd}.} that with high probability, the solution $\gest$ of \eqref{prog:trs} satisfies
\begin{align} \label{eq:cmt_err_bd}
    \norm{\mproj{\gest} - h}_2^2 \lesssim \sigma n + \frac{\reg M^2 \Delta^3}{n}.
\end{align}
On the other hand, one can show that $\norm{z-h}_2^2 \asymp \sigma^2 n$ with high probability if $\sigma \lesssim 1$, hence we cannot conclude   
\begin{align} \label{eq:prov_den}
\rev{\norm{\frac{\gest}{\abs{\gest}} - h}_2} \ll \norm{z-h}_2.
\end{align}
This motivates the present paper which seeks to identify conditions under which \eqref{eq:prov_den} holds. In fact, we will consider a more abstract problem setting where $G$ is any connected graph, and $h \in \calC_n$ is smooth w.r.t $G$. This is formally described below, followed by a summary of our main results.
%
%
\subsection{Problem setup} \label{subsec:prob_setup}
Consider an undirected, connected graph $G = ([n], E)$ where $E \subseteq \set{\set{i,j}: i \neq j \in [n]}$ denotes its set of edges. Denote the maximum degree of $G$ by $\triangle$, and the (combinatorial) Laplacian matrix associated with $G$ by $L \in \matR^{n \times n}$. We will denote the eigenvalues of $L$ by
\begin{equation*}
    \lambda_n = 0 < \lambda_{n-1} = \lambda_{\min} \leq \lambda_{n-2} \leq \cdots \leq \lambda_1
\end{equation*}
where $\lambda_{\min}$ denotes the Fiedler value of $G$, and is well known to be a measure of connectivity of $G$. The corresponding (unit $\ell_2$ norm) eigenvectors of $L$ will be denoted by $q_j \in \matR^n$; $j=1,\dots,n$ where we have that $q_n = \frac{1}{\sqrt{n}}(1,\dots,1)^T$. 

Let $h = [h_1,\dots,h_n]^T \in \calC_n$ be an unknown ground truth signal which is assumed to be smooth w.r.t $G$ in the sense that 
\begin{align} \label{eq:smooth_cond} 
 h^* L h = \sum_{\set{i,j} \in E} \abs{h_i - h_j}^2 \leq \smooth_{n}.
\end{align}
Here, $\smooth_n$ will typically be ``small'' with the subscript $n$ depicting possible dependence on $n$. We will assume that information about $h$ is available  in the form of noisy $z \in \calC_n$. Specifically, denoting $\eta_i \sim \calN(0,\sigma^2)$; $i=1,\dots,n$ to be i.i.d Gaussian random variables, we are given 
\begin{equation} \label{eq:noise_mod}
 z_i = h_i \exp(\iota 2\pi \eta_i); \quad i=1,\dots,n.    
\end{equation}
Given access to the noisy samples $z \in \calC_n$, and the graph $G$, we aim to answer the following two questions.
\begin{enumerate}
\item Consider the unconstrained quadratic program (UCQP) obtained by relaxing the constraints in \eqref{prog:qcqp} to $\mathbb{C}^n$
\begin{equation} 
\min_{g \in \mathbb{C}^n} \norm{g - z}_2^2 + \reg g^* L g. \label{prog:ucqp} \tag{$\text{UCQP}$}
\end{equation}
This is a convex program, also referred to as Tikhonov regularization in the literature (e.g., \cite{shuman13}), and its solution is given in closed form by $\gest = (I + \reg L)^{-1} z$. Under what conditions can we ensure that $\gest$ satisfies \eqref{eq:prov_den}?

\item Under what conditions can we ensure that the solution $\gest$ of \eqref{prog:trs} satisfies \eqref{eq:prov_den}?
\end{enumerate}
As we will see in Section \ref{sec:l2_TRS}, the solution of  \eqref{prog:trs} is closely related to that of \eqref{prog:ucqp} but requires a more careful analysis.
\begin{remark}
\secrev{It is worth mentioning that the quadratic penalty term $\reg g^* L g$ in \eqref{prog:ucqp} and \eqref{prog:trs} aims to promote solutions $\gest$ which are smooth w.r.t $G$. This is natural given that the ground truth signal $h \in \calC_n$ is assumed to be smooth w.r.t. $G$, as stated in \eqref{eq:smooth_cond}. Moreover, the choice of the regularization parameter $\gamma$ is important since larger values of $\gamma$ increase the smoothness of the estimate (larger bias), while smaller values increase the variance of the estimate. For e.g., if we set $\gamma = 0$, which is equivalent to removing the regularization term, then we obtain $\gest = z$. The main point of the ensuing analysis is to find a suitable ``intermediate'' choice of $\gamma$ which ensures that $\gest$ satisfies \eqref{eq:prov_den}.}
\end{remark}

\paragraph{Example: denoising modulo samples of a function.} We will also seek to instantiate the above results for the setup in \cite{CMT18_long} described earlier. More precisely, denoting $f: [0,1] \rightarrow \matR$ to be a $M$-Lipschitz function where $\abs{f(x) - f(y)} \leq M \abs{x-y}$ for each $x,y \in [0,1]$, let $x_i = \frac{i-1}{n-1}$ for $i=1,\dots,n$ be a uniform grid. For each $i$, we are given 
\begin{equation} \label{eq:mod1_noisy_samp_func}
    y_i = (f(x_i) + \eta_i) \mod 1; \quad \eta_i \sim \calN(0,\sigma^2), 
\end{equation}
where $\eta_i$ are i.i.d. Then denoting $z_i = \exp(\iota 2\pi y_i)$ and $h_i = \exp(\iota 2\pi f(x_i))$, we will consider $G = ([n], E)$ to be the \emph{path} graph ($P_n$), where
\begin{equation*}
    E = \set{\set{i,i+1}: i=1,\dots,n}.
\end{equation*}
%
%
\subsection{Main results}
Before stating our results, let us define for any $\lambda \in [\lambda_{\min}, \lambda_1]$, the set
\begin{equation} \label{eq:low_freq_set}
    \calL_{\lambda} := \set{j \in [n-1]: \lambda_j < \lambda}
\end{equation}
consisting of indices corresponding to the ``low frequency'' part of the spectrum of $L$. Moreover, while all our results are non-asymptotic, we suppress constants in this section for clarity. \secrev{The reader is referred to Appendix \ref{appsec:summary_notation} for a tabular summary of the main notation used in the paper.}
%
%
\paragraph{Results for \eqref{prog:ucqp}.} 
We begin by outlining our main results for the estimator \eqref{prog:ucqp}. Theorem \ref{thm:ucqp_denoise_expec_main} below identifies conditions on $\sigma$ under which \eqref{prog:ucqp} provably denoises the samples in expectation. The statement is a simplified version of Theorem \ref{thm:ucqp_denoise_expec}.
\begin{theorem} \label{thm:ucqp_denoise_expec_main}
For any given $\varepsilon \in (0,1)$ and $\lambbar \in [\lambmin, \lambda_1]$
satisfying $1 + \abs{\calL_{\lambbar}} \lesssim \varepsilon n$, suppose that $\frac{1}{\varepsilon \lambbar} \sqrt{\frac{{\triangle \smooth_n}}{n}} \lesssim \sigma \lesssim 1$. Then for the choice $\reg \asymp (\frac{\sigma^2 n}{\triangle \smooth_n \lambbar^2})^{1/4}$, the solution $\gest$ of \eqref{prog:ucqp} satisfies $\expec \norm{\frac{\gest}{\abs{\gest}} - h}_2^2 \leq \varepsilon \expec \norm{z - h}_2^2.$
\end{theorem}
The parameter $\lambbar$ depends on the spectrum of the Laplacian and can be thought of as the ``cut-off frequency''. As can be seen from Theorem \ref{thm:ucqp_denoise_expec_main}, we would ideally like $\lambbar$ to be ``large'' and $\abs{\calL_{\lambbar}}$ to be ``small''; in particular, $\abs{\calL_{\lambbar}} = o(n)$. For e.g., when $G$ is the complete graph\footnote{\rev{Recall that in the complete graph, there is an edge $\set{i,j}$ for each $i \neq j$.}} ($K_n$), then $\lambda_n = 0$ and $\lambda_i = n$ for $i=1,\dots,n-1$. Hence, the choice $\lambbar = n$ is ideal as it leads to 
$\abs{\calL_{\lambbar}} = 0$. Furthermore, the noise regime in the Theorem involves a lower bound on $\sigma$ which might seem unnatural. We believe this is likely an artefact of the analysis involving the estimation error. Specifically, the error bound we obtain is of the form (see Lemma \ref{eq:ucqp_bd_1})
\begin{equation*}
\expec \norm{\frac{\gest}{\abs{\gest}} - h}_2^2 
    \lesssim  \frac{\sigma}{\lambbar} \sqrt{\triangle \smooth_n n} + \sigma^2(1 + \abs{\calL_{\lambbar}}).
\end{equation*}
The problematic term is the first term which depends linearly on $\sigma$. Indeed, since $\norm{z-h}_2^2 \asymp \sigma^2 n$ w.h.p, we end up with the lower bound requirement when $\sigma \ll 1$ for ensuring the denoising guarantee. Nevertheless, if $\frac{1}{\varepsilon \lambbar} \sqrt{\frac{{\triangle \smooth_n}}{n}} = o(1)$ as $n$ increases, the requirement on $\sigma$ is of the form $o(1) \leq \sigma \lesssim 1$ which becomes progressively mild as $n$ increases.

We also derive conditions under which denoising occurs with high probability\footnote{probability approaching $1$ as $n \rightarrow \infty$.} (w.h.p). Theorem \ref{thm:ucqp_denoise_prob_main} below is a simplified version of Theorem \ref{thm:ucqp_denoise_prob}.
\begin{theorem} \label{thm:ucqp_denoise_prob_main}
For any given $\varepsilon \in (0,1)$ and $\lambbar \in [\lambmin, \lambda_1]$, suppose that 
\begin{equation*}
    \max \set{\frac{1}{\varepsilon \lambbar} \sqrt{\frac{{\triangle \smooth_n}}{n}}, \frac{\log n}{\sqrt{\varepsilon n}}} \lesssim 
    \sigma \lesssim 1 \quad \text{ and }  \quad 
    1+\abs{\calL_{\lambbar}} + \sqrt{(1 + \abs{\calL_{\lambbar}}) \log n} \lesssim  \varepsilon n.
\end{equation*}
If $\reg \asymp (\frac{\sigma^2 n}{\triangle \smooth_n \lambbar^2})^{1/4}$, then w.h.p, the solution $\gest$ of \eqref{prog:ucqp} satisfies $\norm{\frac{\gest}{\abs{\gest}} - h}_2^2 \leq \varepsilon \norm{z - h}_2^2$.
\end{theorem}
The conditions in Theorem \ref{thm:ucqp_denoise_prob_main} are slightly more stringent compared to those in Theorem \ref{thm:ucqp_denoise_expec_main} due to the appearance of extra $\log$ factors. These arise due to the concentration inequalities used in the analysis for bounding the estimation error, see Theorem \ref{thm:ucqp_prob_bd}.

In Section \ref{sec:l2_UCQP}, we interpret our results for special graphs\footnote{\secrev{The choice of the graphs $K_n, P_n$ and $S_n$ is only for convenience, one could consider other connected graphs as well.}} such as $K_n, P_n$ and the star graph\footnote{\rev{Recall that a star graph is a tree with one vertex having degree $n-1$, and the remaining $n-1$ vertices having degree $1$.}} ($S_n$); see  Corollaries \ref{cor:ucqp_den_expec}, \ref{cor:ucqp_den_prob}. In particular, the statement for $P_n$ therein can be applied to the example from Section \ref{subsec:prob_setup}, readily yielding conditions that ensure denoising; see Corollary \ref{cor:ucqp_den_func_mod}. It states that if  $\frac{\log n}{\sqrt{n}} \lesssim \sigma \lesssim 1$ and $n \gtrsim 1$, then for $\reg \asymp \left(\frac{\sigma^2 n^{10/3}}{M^2} \right)^{1/4}$ the solution $\gest$ of \eqref{prog:ucqp} satisfies (w.h.p)
    \begin{equation*}
       \norm{\frac{\gest}{\abs{\gest}} - h}_2^2 \lesssim (\sigma M + \sigma^2) n^{2/3} + \log n.
    \end{equation*}
Hence for any $\varepsilon \in (0,1)$, in the noise regime $\max\set{\frac{M}{\varepsilon n^{1/3}}, \frac{\log n}{ \sqrt{\varepsilon n}} } \lesssim \sigma \lesssim 1$, if $n \gtrsim (1/\varepsilon)^3$, then \eqref{prog:ucqp} denoises $z$ w.h.p.

%
\paragraph{Results for \eqref{prog:trs}.} We now describe conditions under which the estimator \eqref{prog:trs} provably denoises $z \in \calC_n$ w.h.p. Theorem \ref{thm:trs_prob_denoise_main} below is a simplified version of Theorem \ref{thm:trs_prob_denoise}. 
%
%
\begin{theorem} \label{thm:trs_prob_denoise_main}
For any  $\varepsilon \in (0,1)$,  given $k \in [n-1]$ s.t $\lambda_{n-k+1} < \lambda_{n-k}$  and $\lambbar \in [\lambmin, \lambda_1]$ with the choice $\reg \asymp (\frac{\sigma^2 n}{\triangle \smooth_n \lambbar^2})^{1/4}$, suppose that the following conditions are satisfied.

\begin{enumerate}[label=\upshape(\roman*)]
\item $\smooth_n \lesssim  \min\set{n \lambda_{n-k}, n \lambbar}$,  and $1 + \abs{\calL_{\lambbar}} + \sqrt{(1 + \abs{\calL_{\lambbar}}) \log n} \lesssim \varepsilon n$.
    
 \item $\sigma \lesssim
\min\set{\sqrt{\varepsilon}, \frac{\lambbar}{\lambda_{n-k+1}^2} \sqrt{\frac{\triangle \smooth_n}{n}}}$ and 
    \begin{align*}
        \sigma \gtrsim \max \left\{\left(\frac{\log n}{\sqrt{n}} \right)^{1/2}, \frac{\smooth_n}{n \lambda_{n-k} \sqrt{\varepsilon}}, \sqrt{\frac{\log n}{n \varepsilon}}, \frac{1}{\varepsilon \lambbar} \left(\sqrt{\frac{\triangle \smooth_n}{n}} + \lambda_{n-k+1}^2 \sqrt{\frac{n}{\triangle \smooth_n}} \right) \right\}. 
    \end{align*}
\end{enumerate}    
Then the (unique) solution $\gest \in \mathbb{C}^n$ of \eqref{prog:trs} satisfies $\norm{\frac{\gest}{\abs{\gest}} - h}_2^2 \leq \varepsilon \norm{z - h}_2^2$ w.h.p.
\end{theorem}
The above statement is admittedly more convoluted than that for $\eqref{prog:ucqp}$ which is primarily due to the more intricate nature of the estimation error analysis. An interesting aspect of the above result is that we can consider the ``best'' choice of $k$ (satisfying $\lambda_{n-k+1} < \lambda_{n-k}$) which leads to the mildest constraints on $\sigma$ and $\smooth_n$. Note that since $G$ is connected (by assumption), $k=1$ always satisfies this condition ($0 = \lambda_n < \lambda_{n-1}$). This leads to the following useful Corollary which is a simplified version of Corollary \ref{cor:trs_high_prob_simp_den}.
%
%
%
\begin{corollary} \label{cor:trs_high_prob_simp_den_main}
For any  $\varepsilon \in (0,1)$ and $\lambbar \in [\lambmin, \lambda_1]$ with the choice $\reg \asymp (\frac{\sigma^2 n}{\triangle \smooth_n \lambbar^2})^{1/4}$, suppose that the following conditions are satisfied.

\begin{enumerate}[label=\upshape(\roman*)]
    \item  $\smooth_n \lesssim  n \lambmin$ and $1 + \abs{\calL_{\lambbar}} + \sqrt{(1 + \abs{\calL_{\lambbar}}) \log n} \lesssim  \varepsilon n$.
    
     \item   
    %
    %
      $   \max\set{ \left(\frac{\log n}{\sqrt{n}} \right)^{1/2},  \frac{\smooth_n}{n \lambmin \sqrt{\varepsilon}}, \sqrt{\frac{\log n}{n}}, 
        \frac{1}{\varepsilon \lambbar}  \sqrt{\frac{\triangle \smooth_n}{n}}} \lesssim \sigma \lesssim \sqrt{\varepsilon}.$
%
\end{enumerate}    
Then the (unique) solution $\gest \in \mathbb{C}^n$ of \eqref{prog:trs} satisfies $\norm{\frac{\gest}{\abs{\gest}} - h}_2^2 \leq \varepsilon \norm{z - h}_2^2$ w.h.p.
\end{corollary}
It is natural to ask whether Corollary \ref{cor:trs_high_prob_simp_den_main} is -- for all practical purposes -- sufficient, compared to Theorem \ref{thm:trs_prob_denoise_main}. For the complete graph $K_n$, observe that the only valid choice is $k=1$, hence we need Corollary \ref{cor:trs_high_prob_simp_den_main}. However, as we will see in Section \ref{sec:l2_TRS}, it turns out that for the path graph $P_n$, Corollary \ref{cor:trs_high_prob_simp_den_main} leads to unreasonably strict conditions on $\sigma$ and $\smooth_n$ (see Remark \ref{rem:trs_exgraph_large_n}). In fact for the path graph, $\lambda_{n-k+1} < \lambda_{n-k}$ holds for each $k \in [n-1]$, hence a careful choice of $k$ in Theorem \ref{thm:trs_prob_denoise_main} is key to obtaining satisfactory conditions, see Corollary \ref{cor:Pn_res_bet_conds}. 

In the setting of the example from Section \ref{subsec:prob_setup}, Corollary \ref{cor:Pn_res_bet_conds} roughly states that for any $\varepsilon \in (0,1)$, in the noise regime $(\frac{\log n}{\sqrt{n}})^{1/2} \lesssim \sigma \lesssim \sqrt{\varepsilon}$ with $n$ large enough and a suitably chosen $\reg$, \eqref{prog:trs} denoises $z$ w.h.p. This condition on $\sigma$ is visibly stricter than that for \eqref{prog:ucqp}; we believe that this is likely an artefact of the analysis. Nevertheless, for this example, we see for $\varepsilon$ fixed and $n \rightarrow \infty$ that both \eqref{prog:trs} and \eqref{prog:ucqp} provably denoise $z$ w.h.p in the noise regime $o(1) \leq \sigma \lesssim 1$.
\paragraph{Outline of the paper.} The rest of the paper is organized as follows. Section \ref{sec:prelim} introduces some preliminaries involving certain intermediate facts and technical results that will be needed in our analysis. Section \ref{sec:l2_UCQP} contains the analysis for \eqref{prog:ucqp} while Section \ref{sec:l2_TRS} analyzes  \eqref{prog:trs}. \rev{Section \ref{sec:sims} contains some simulation results on synthetic examples for \eqref{prog:ucqp} and \eqref{prog:trs}.} We conclude with Section \ref{sec:discussion} which contains a discussion with related work from the literature, and directions for future work. 

\section{Preliminaries} \label{sec:prelim}
This section summarizes some useful technical tools that will be employed at multiple points in our analysis. We begin by deriving some simple consequences of the smoothness assumption in \eqref{eq:smooth_cond}.
%
\paragraph{Smooth phase signals.} Similar to \eqref{eq:low_freq_set}, for $\lambda \in [\lambda_{\min}, \lambda_1]$, let us define the set
\begin{equation*}
    \calH_{\lambda} := \set{j \in [n-1]: \lambda_j \geq \lambda}
\end{equation*}
consisting of indices corresponding to the ``high frequency'' part of the spectrum of $L$. Then using \eqref{eq:smooth_cond} and the fact $\norm{h}_2^2 = n$, it is not difficult to establish that 
\begin{equation} \label{eq:smooth_res1}
  \sum_{j \in \calH_{\lambda}} \abs{\dotprod{h}{q_j}}^2 \leq \frac{\smooth_n}{\lambda}, \qquad \sum_{j \in  \calL_{\lambda} \cup \set{n}} \abs{\dotprod{h}{q_j}}^2 \geq n - \frac{\smooth_n}{\lambda}.    
\end{equation}
Hence the smaller $\smooth_n$ is, the more correlated $h$ is with the eigenvectors corresponding to $\set{n} \cup \calL_{\lambda}$. It is also useful to note that since 
$$\abs{(L h)_i}^2 = \abs{\sum_{j:\set{i,j} \in E} (h_i - h_j)}^2 \leq \triangle \sum_{j:\set{i,j} \in E} \abs{h_i - h_j}^2, $$ hence $\norm{L h}_2^2 \leq 2 \triangle h^* L h \leq 2\triangle \smooth_n$ which is equivalent to 
\begin{equation} \label{eq:smooth_res2}
   \sum_{j = 1}^{n-1} \lambda_j^2 \abs{\dotprod{h}{q_j}}^2  \leq 2 \triangle \smooth_n.
\end{equation}
Finally, recall the setup in Section \ref{subsec:prob_setup} where we obtain noisy modulo samples of a smooth function $f$ on a uniform grid (with $G = P_n$). In this setting, we can relate $\smooth_n$ to the quadratic variation of $f$ on the grid. Indeed, using the fact $\abs{h_i - h_{i+1}} \leq 2 \pi \abs{f(x_i) - f(x_{i+1})}$ (see proof of \cite[Lemma 8]{CMT18_long}), it follows that
\begin{equation} \label{eq:func_quad_var_Bn}
    \sum_{i=1}^{n-1} \abs{h_i - h_{i+1}}^2 \leq 4\pi^2 \sum_{i=1}^{n-1} \abs{f(x_i) - f(x_{i+1})}^2.
\end{equation}

\paragraph{Noise model.} 
Next, we collect some useful results related to the random noise model in \eqref{eq:noise_mod}. The following Proposition is easy to verify, its proof is provided in the appendix for completeness.
\begin{proposition} \label{prop:noise_expec_iden}
Denote $z = [z_1,\dots, z_n]^T \in \calC_n$ with $z_i$ as defined in \eqref{eq:noise_mod}. Let $u \in \mathbb{C}^n$ be given with $\norm{u}_2 = 1$. Then the following is true.
\begin{enumerate}[label=\upshape(\roman*)]
    \item \label{item:noise_expec_iden_1} $\expec[z] = e^{-2\pi^2 \sigma^2} h$.

    \item \label{item:noise_expec_iden_2} $\expec\left[\abs{\dotprod{z- e^{-2\pi^2 \sigma^2} h}{u}}^2 \right] = (1- e^{-4\pi^2 \sigma^2})$.
    
    \item \label{item:noise_expec_iden_3} $\expec[\abs{\dotprod{z}{u}}^2] = e^{-4\pi^2 \sigma^2} \abs{\dotprod{h}{u}}^2 + (1- e^{-4\pi^2 \sigma^2})$.
    
    \item \label{item:noise_expec_iden_4} $\expec \left[\norm{z -\expn{2} h}_2^2 \right] = n (1- e^{-4\pi^2 \sigma^2})$.
    
    \item \label{item:noise_expec_iden_5} $\expec[\norm{z-h}_2^2] = 2n (1-e^{-2\pi^2 \sigma^2})$.
\end{enumerate}
In particular, if $\sigma \leq \frac{1}{\pi \sqrt{2}}$ then the expected distance of $z$ from $h$ can be bounded as
\begin{equation} \label{eq:expec_zh_dist_bds}
    2\pi^2 \sigma^2 n \leq \expec[\norm{z-h}_2^2] \leq 4\pi^2 \sigma^2 n.
\end{equation}
\end{proposition}
\begin{proof}
See Appendix \ref{appsec:proof_prop_noise_expec}
\end{proof}
We will also require several concentration bounds in order to derive high probability error bounds later on in our analysis. 
\begin{proposition} \label{prop:conc_bounds}
Denote $z = [z_1,\dots, z_n]^T \in \calC_n$ with $z_i$ as defined in \eqref{eq:noise_mod}. For any given $U = [u_1 \cdots u_k] \in \matR^{n \times k}$ with orthonormal columns, the following holds true.
\begin{enumerate}[label=\upshape(\roman*)]
    \item\label{prop:conc_bds_1} With probability at least $1 - \frac{4}{n^2}$,  
    \begin{align*}
  & z^{*} UU^T z - k (1-\expn{4}) - \expn{4} h^* UU^T h \\
    &\geq -\left[ 4096 \log n +32\sqrt{6k}  (1-\expn{8}) \sqrt{\log n}     
     + 11 \log n \left(\norm{UU^T h}_{\infty} + \sqrt{1-\expn{8}} \norm{UU^T h}_2 \right) \right].
    \end{align*}

    \item\label{prop:conc_bds_2} With probability at least $1 - \frac{2}{n^2}$, 
       \begin{align*}
     (z-\expn{2} h)^* UU^T (z - \expn{2} h) \leq k(1 - \expn{4}) + 4096 \log n + 32\sqrt{6k}  (1 - \expn{8}) \sqrt{\log n}.   
    \end{align*} 
    
    \item\label{prop:conc_bds_3} With probability at least $1 - \frac{2}{n^2}$,
    \begin{equation*}
     \norm{z - h}_2^2 - 2n (1-\expn{2}) \leq  3 \log n \left(2 + \sqrt{4 + 9(1-\expn{8}) n} \right).  
    \end{equation*}
    
    \item \label{prop:conc_bds_4} With probability at least $1 - \frac{2}{n^2}$,
    
    \begin{equation*}
     \norm{z - \expn{2} h}_2^2 - n (1-\expn{4}) \leq  3 \log n \left(2 + \sqrt{4 + 9(1-\expn{8}) n} \right).  
    \end{equation*}
\end{enumerate}
In all the above inequalities, the event obtained by reversing both the inequality and the sign of the RHS term also holds with the stated probability of success.
\end{proposition}
\begin{proof}
See Appendix \ref{appsec:proof_prop_conc}.
\end{proof}
%
%
\paragraph{Entrywise projection on $\calC_n$.}
Lastly, we will make extensive use of the following useful result from \cite[Proposition 3.3]{Liu2017} concerning the projection operator in \eqref{eq:project_manifold}.
\begin{proposition}[\cite{Liu2017}] \label{prop:entry_proj}
For any $q \in [1,\infty]$, $w  \in \mathbb{C}^n$ and $g \in \calC_n$, we have
\begin{equation*}
    \norm{\mproj{w} - g}_q \leq 2 \norm{w - g}_q.
\end{equation*}
\end{proposition}
It is especially important to note that for any number $t > 0$, we have $\mproj{w} = \mproj{t w}$. Hence from the above Proposition, it follows that $\norm{\mproj{w} - g}_q \leq 2 \norm{t w - g}_q$ for any $t > 0$. While it might be difficult to choose the optimal scale $t > 0$ that minimizes the above bound, one could still choose a good surrogate that leads to an improvement over the bound obtained when $t = 1$. 

\section{Error bounds for \eqref{prog:ucqp}} \label{sec:l2_UCQP}
We now analyse the quality of the solution $\gest = (I + \reg L)^{-1} z$ of \eqref{prog:ucqp}. To begin with, we have the following Lemma that bounds the error between $\frac{\gest}{\abs{\gest}}$ and $h$ \emph{for any} $z \in \calC_n$.
\begin{lemma} \label{lem:ucqp_lem_err_1}
For any $z \in \calC_n$ and $\lambbar \in [\lambmin, \lambda_1]$, it holds that 
\begin{equation*}
    \norm{\frac{\gest}{\abs{\gest}} - h}_2^2 \leq 8(E_1 + E_2)
\end{equation*}
where 
\begin{align*} 
E_1 &= \abs{\dotprod{q_n}{\expp{2}z - h}}^2 + \frac{1}{(1+\reg\lambmin)^2} \sum_{j \in \calL_{\lambbar}} \abs{\dotprod{q_j}{\expp{2}z-h}}^2 + \frac{\norm{\expp{2} z - h}_2^2}{(1+\reg\lambbar)^2}, \\
E_2 &= \frac{2 \reg^2 \triangle \smooth_n}{(1 + \reg \lambmin)^2}.
\end{align*}
\end{lemma}
\begin{proof}
From the definition in \eqref{eq:project_manifold}, observe that $\mproj{\gest} = \mproj{\expp{2} \gest}$ where we recall $\gest = (I + \reg L)^{-1} z$. We can write
\begin{equation*}
    \expp{2} \gest - h = \underbrace{(I + \reg L)^{-1} (\expp{2} z - h)}_{e_1} + \underbrace{(I + \reg L)^{-1}h - h}_{e_2}
\end{equation*}
which implies $\norm{\expp{2} \gest - h}_2^2 \leq 2(\norm{e_1}_2^2 + \norm{e_2}_2^2)$. Using Proposition \ref{prop:entry_proj}, we then obtain 
\begin{align*}
    \norm{\frac{\gest}{\abs{\gest}} - h}_2^2 = 
    \norm{\frac{\expp{2} \gest}{\abs{\expp{2} \gest}} - h}_2^2
    \leq 4 \norm{\expp{2} \gest - h}_2^2 \leq 8(\norm{e_1}_2^2 + \norm{e_2}_2^2).
\end{align*}
%
%
%
\paragraph{Bounding $\norm{e_2}_2^2$.} Using $h = \sum_{j=1}^n \dotprod{q_j}{h} q_j$, we can write $e_2$ as
\begin{equation*}
e_2 = (I + \reg L)^{-1}h - h = \sum_{j=1}^{n} \frac{\dotprod{q_j}{h}}{1+\reg \lambda_j} q_j - \sum_{j=1}^{n} \dotprod{q_j}{h} q_j =  \sum_{j=1}^{n-1} \left(\frac{-\reg \lambda_j}{1 + \reg \lambda_j} \right) \dotprod{q_j}{h} q_j.     
\end{equation*}
This leads to the bound (using orthonormality of $q_j$'s)
\begin{equation*}
 \norm{e_2}_2^2 
 =  \sum_{j=1}^{n-1} \left(\frac{\reg^2 \lambda_j^2}{(1 + \reg \lambda_j)^2} \right) \abs{\dotprod{h}{q_j}}^2  
 \leq \frac{ \reg^2}{(1 + \reg \lambmin)^2} \sum_{j=1}^{n-1} \lambda_j^2 \abs{\dotprod{h}{q_j}}^2 
 \leq \frac{2 \triangle \reg^2 \smooth_n}{(1 + \reg \lambmin)^2}
\end{equation*}
where the last inequality uses \eqref{eq:smooth_res2}.

%
\paragraph{Bounding $\norm{e_1}_2^2$.} By writing $e_1$ as
\begin{equation*}
    e_1 = (I + \reg L)^{-1} (\expp{2} z - h) = \dotprod{q_n}{\expp{2} z-h} q_n + \sum_{j=1}^{n-1} \frac{\dotprod{q_j}{\expp{2} z-h}}{1 + \reg \lambda_j} q_j
\end{equation*}
we obtain the bound
\begin{align*}
    \norm{e_1}_2^2 
    &= \abs{\dotprod{q_n}{\expp{2} z-h}}^2 + \sum_{j=1}^{n-1} \frac{\abs{\dotprod{q_j}{\expp{2} z-h}}^2}{(1 + \reg \lambda_j)^2} \\
    &\leq  \abs{\dotprod{q_n}{\expp{2} z-h}}^2 + \sum_{j \in \calL_{\lambbar}}  \frac{\abs{\dotprod{q_j}{\expp{2} z-h}}^2}{(1 + \reg \lambmin)^2} + 
    \sum_{j \in \calH_{\lambbar}}  \frac{\abs{\dotprod{q_j}{\expp{2} z-h}}^2}{(1 + \reg \lambbar)^2} \\
    &= \abs{\dotprod{q_n}{\expp{2} z-h}}^2 \left(1 - \frac{1}{(1 + \reg \lambbar)^2} \right) 
    + \left( \frac{1}{(1 + \reg \lambmin)^2} - \frac{1}{(1 + \reg \lambbar)^2} \right) \sum_{j \in \calL_{\lambbar}} \abs{\dotprod{q_j}{\expp{2} z-h}}^2 \\ 
    &+ \frac{\norm{\expp{2} z - h}_2^2}{(1+\reg\lambbar)^2} \\
    &\leq \abs{\dotprod{q_n}{\expp{2} z-h}}^2 +  \frac{1}{(1 + \reg \lambmin)^2} \sum_{j \in \calL_{\lambbar}} \abs{\dotprod{q_j}{\expp{2} z-h}}^2 + \frac{\norm{\expp{2} z - h}_2^2}{(1+\reg\lambbar)^2}
\end{align*}
where in the penultimate step, we used the identity 
$$\sum_{j \in \calH_{\lambbar}} \abs{\dotprod{q_j}{\expp{2} z-h}}^2 = \norm{\expp{2} z - h}_2^2 - \sum_{j \in \calL_{\lambbar} \cup \set{n}} \abs{\dotprod{q_j}{\expp{2} z-h}}^2.$$
\end{proof}
The quantity $\lambbar$ depends on the graph $G$, and as we will see shortly, we will like to choose $\lambbar$ such that (a) $\abs{\calL_{\lambbar}}$ is small and, (b) $\sum_{j \in \calL_{\lambbar} \cup \set{n}} \abs{\dotprod{h}{q_j}}^2 \approx \norm{h}_2^2 = n$.
\subsection{Error bounds in expectation} \label{subsec:ucqp_err_bds_expec}
If $z \in \calC_n$ is generated randomly as in \eqref{eq:noise_mod}, we easily obtain from Lemma \ref{lem:ucqp_lem_err_1} a bound on the expected $\ell_2$ error.
\begin{lemma} \label{lem:ucqp_lem_err_2}
Let $z \in \calC_n$ be generated as in \eqref{eq:noise_mod} and assume $\sigma \leq \frac{1}{2\pi}$. Then for any given $\lambbar \in [\lambmin, \lambda_1]$, it holds that 
\begin{eqnarray} \label{eq:ucqp_bd_1}
    \expec \norm{\frac{\gest}{\abs{\gest}} - h}_2^2 
    \leq 
    \frac{16 \triangle \reg^2 \smooth_n}{(1 + \reg \lambmin)^2} 
    + 64\pi^2 \sigma^2 \left(1 +  \frac{\abs{\calL_{\lambbar}}}{(1 + \reg \lambmin)^2}  + \frac{n}{(1 + \reg \lambbar)^2}\right).   
\end{eqnarray}
In particular, if $\reg = (\frac{4\pi^2 \sigma^2 n}{\triangle \smooth_n \lambbar^2})^{1/4}$, we obtain the simplified bound
\begin{eqnarray} \label{eq:ucqp_sim_bd}
    \expec \norm{\frac{\gest}{\abs{\gest}} - h}_2^2 
    \leq  64\pi \left(\frac{\sigma}{\lambbar} \sqrt{\triangle \smooth_n n} + \pi \sigma^2(1 + \abs{\calL_{\lambbar}}) \right).
\end{eqnarray}
\end{lemma}
\begin{proof}
By taking the expectation of both sides of the inequality in Lemma \ref{lem:ucqp_lem_err_1}, and using Proposition \ref{prop:noise_expec_iden}, we obtain 
\begin{eqnarray*}
    \expec \norm{\frac{\gest}{\abs{\gest}} - h}_2^2 
    \leq   
    \frac{16 \triangle \reg^2 \smooth_n}{(1 + \reg \lambmin)^2} 
    + 8 (e^{4\pi^2 \sigma^2} - 1) \left(1 +  \frac{\abs{\calL_{\lambbar}}}{(1 + \reg \lambmin)^2}  + \frac{n}{(1 + \reg \lambbar)^2}\right). 
\end{eqnarray*}
The bound in \eqref{eq:ucqp_bd_1} now follows from the standard fact $e^{x} \leq 1 + 2x$ for any $x \leq 1$.

To obtain the bound in \eqref{eq:ucqp_sim_bd}, we first use $1 + \reg \lambmin \geq 1$ and $1 + \reg \lambbar \geq \reg \lambbar$ in \eqref{eq:ucqp_bd_1} and observe that  the  resulting  bound is a convex function of $\reg$ minimized for the stated value of $\reg$. Plugging this choice of $\reg$ in \eqref{eq:ucqp_bd_1} then yields \eqref{eq:ucqp_sim_bd}.
\end{proof}
As can be seen from \eqref{eq:ucqp_sim_bd}, there is a trade-off in the choice of $\lambbar$.  We are now ready to state our first main theorem which establishes conditions under which \eqref{prog:ucqp} provably denoises the input $z \in \calC_n$. It follows easily from Lemma \ref{lem:ucqp_lem_err_2} and Proposition \ref{prop:noise_expec_iden}\ref{item:noise_expec_iden_5}.
\begin{theorem} \label{thm:ucqp_denoise_expec}
Let $z \in \calC_n$ be generated as in \eqref{eq:noise_mod}, and $\varepsilon \in (0,1)$. For any given $\lambbar \in [\lambmin, \lambda_1]$
satisfying $1 + \abs{\calL_{\lambbar}} \leq \frac{\varepsilon n}{64}$, suppose that the noise level $\sigma$ satisfies
\begin{equation*}
    \frac{64}{\pi \varepsilon \lambbar} \sqrt{\frac{{\triangle \smooth_n}}{n}} \leq \sigma \leq \frac{1}{2 \pi}. 
\end{equation*}
Then for the choice $\reg = (\frac{4\pi^2 \sigma^2 n}{\triangle \smooth_n \lambbar^2})^{1/4}$, the solution $\gest$ of \eqref{prog:ucqp} satisfies
\begin{equation} \label{eq:ucqp_eps_den_bd}
\expec \norm{\frac{\gest}{\abs{\gest}} - h}_2^2 \leq \varepsilon \expec \norm{z - h}_2^2.    
\end{equation}
\end{theorem}
 \begin{proof}
Recall from \eqref{eq:expec_zh_dist_bds} that if $\sigma \leq \frac{1}{\pi \sqrt{2}}$, then $\expec\norm{z - h}_2^2 \geq 2\pi^2 \sigma^2 n$ holds. Then using \eqref{eq:ucqp_sim_bd} from Lemma \ref{lem:ucqp_lem_err_2}, we see that if $\sigma \leq \frac{1}{2\pi}$, then for any $\varepsilon \in (0,1)$, the bound in \eqref{eq:ucqp_eps_den_bd} is ensured provided 
\begin{equation} \label{ucqp_mainthm_tmp1}
32 \left(\frac{\sqrt{\triangle \smooth_n n}}{\lambbar}  + \pi \sigma(1 + \abs{\calL_{\lambbar}}) \right) \leq \pi \varepsilon \sigma n.   
\end{equation}
Finally, \eqref{ucqp_mainthm_tmp1} is clearly ensured provided 
$1 + \abs{\calL_{\lambbar}} \leq \frac{\varepsilon n}{64}$ and  
 \begin{equation*}
     32 \frac{\sqrt{\triangle \smooth_n n}}{\lambbar} \leq \frac{\sigma \varepsilon n}{2} \iff \frac{64}{\varepsilon \lambbar} \sqrt{\frac{{\triangle \smooth_n}}{n}} \leq 
    \sigma
 \end{equation*}
 which completes the proof.
\end{proof}
The following remarks are in order.
\begin{enumerate}[label=\upshape(\roman*)]
    \item The condition $1 + \abs{\calL_{\lambbar}} \lesssim \varepsilon n$ in Theorem \ref{thm:ucqp_denoise_expec} implies that we require $\abs{\calL_{\lambbar}} = o(n)$ as $n$ increases, for fixed $\varepsilon$. In other words, $\lambbar$ should not be chosen to be too large.
    
    \item The Theorem requires that the noise level lies in the regime $\frac{1}{\varepsilon \lambbar} \sqrt{\frac{{\triangle \smooth_n}}{n}} \lesssim \sigma \lesssim 1$ for denoising. The lower bound therein is likely due to an artefact of the analysis, and arises from the first error term in \eqref{eq:ucqp_sim_bd} which scales as $\frac{\sigma}{\lambbar} \sqrt{\triangle \smooth_n n}$. Nevertheless, if  
    \begin{equation*}
        \frac{1}{\lambbar} \sqrt{\frac{{\triangle \smooth_n}}{n}} = o(1) \quad \text{ as } n \rightarrow \infty,
    \end{equation*}
    then the condition on $\sigma$ weakens considerably as $n$ increases.
\end{enumerate}

It is interesting to translate the conditions of Theorem \ref{thm:ucqp_denoise_expec} for special choices of $G$, namely $K_n$ (complete graph), $S_n$ (star graph) and $P_n$ (path graph). This leads to the following Corollary. 
\begin{corollary} \label{cor:ucqp_den_expec}
 Let $z \in \calC_n$ be generated as in \eqref{eq:noise_mod}, and $\varepsilon \in (0,1)$. 
 \begin{enumerate}[label=\upshape(\roman*)]
     \item\label{cor:ucqp_den_eps_Kn} {($G = K_n$)} Suppose $n \gtrsim \frac{1}{\varepsilon}$ and $\frac{1}{n \varepsilon} \sqrt{\smooth_n} \lesssim \sigma \lesssim 1$.  If $\reg \asymp (\frac{\sigma^2}{n^2 \smooth_n })^{1/4}$, then the solution $\gest$ of \eqref{prog:ucqp} satisfies \eqref{eq:ucqp_eps_den_bd}.    
     
     \item\label{cor:ucqp_den_eps_Sn} {($G = S_n$)} Suppose $n \gtrsim \frac{1}{\varepsilon}$ and $\frac{1}{\varepsilon} \sqrt{\smooth_n} \lesssim \sigma \lesssim 1$.  If $\reg \asymp (\frac{\sigma^2}{\smooth_n })^{1/4}$, then the solution $\gest$ of \eqref{prog:ucqp} satisfies \eqref{eq:ucqp_eps_den_bd}.
     
     \item\label{cor:ucqp_den_eps_Pn} {($G = P_n$)} For a given $\theta \in [0,1)$, suppose $n \gtrsim (\frac{1}{\varepsilon})^{\frac{1}{1-\theta}}$ and $\frac{n^{\frac{3}{2} -2\theta}}{\varepsilon} \sqrt{\smooth_n} \lesssim \sigma \lesssim 1$.  If $\reg \asymp (\frac{\sigma^2 n^{5-4\theta}}{\smooth_n })^{1/4}$, then the solution $\gest$ of \eqref{prog:ucqp} satisfies \eqref{eq:ucqp_eps_den_bd}. 
     
 \end{enumerate}
\end{corollary}
\begin{proof}
The spectra of $L$ for $K_n$, $S_n$ and $P_n$ are well known, see for e.g. \cite[Chapter 1]{brouwer12}.
\begin{enumerate}[label=\upshape(\roman*)]
    \item Here, $\triangle = n-1$ and $\lambda_{n-1} = \cdots = \lambda_1 = n$. Choose $\lambbar = n$ so that $\calL_{\lambbar} = \emptyset$.
    
    \item Here, $\triangle = n-1$ while $\lambda_{n-1} = \cdots = \lambda_2 = 1$ and $\lambda_1 = n$. Choose $\lambbar = 1$ so that $\calL_{\lambbar} = \emptyset$.
    
    \item Here, $\lambda_j = 4(\sin^2[\frac{\pi}{2n} (n-j)])$ for $j=1,\dots,n$; hence 
    $$ \lambda_1 \asymp 1 \quad \text{ and } \quad  \lambmin \asymp \sin^2\left(\frac{\pi}{2n}\right) \asymp \frac{1}{n^2}.$$  
    
    Choosing $\lambbar \asymp \frac{1}{n^{2(1-\theta)}}$ for $\theta \in [0,1)$, it is not difficult to see that $\abs{\calL_{\lambbar}} \lesssim n^{\theta}$, and so, $1 + \abs{\calL_{\lambbar}} \lesssim \varepsilon n$ provided $n \gtrsim (\frac{1}{\varepsilon})^{\frac{1}{1-\theta}}$. The conditions on $\sigma$ and $\reg$ follow easily using $\triangle = 2$.
\end{enumerate}
\end{proof}
\begin{remark} \label{rem:cor_ex_graph_expec_ucqp}
For the graphs in Corollary \ref{cor:ucqp_den_expec}, we can deduce conditions on the smoothness term $\smooth_n$ which lead to a non-vacuous regime for $\sigma$, of the form $o(1) \leq \sigma \lesssim 1$. These are detailed below when $n \rightarrow \infty$.
\begin{enumerate}
    \item {($G = K_n$)} If $\varepsilon$ is fixed, then $\smooth_n = o(n^2)$ suffices. On the other hand, if $\varepsilon = \varepsilon_n \rightarrow 0$ (as $n \rightarrow \infty$) and $\varepsilon_n = \omega(\frac{1}{n})$, then $\smooth_n = o(\varepsilon_n^2 n^2)$ is sufficient. 
    
    \item {($G = S_n$)} For $\varepsilon$ fixed, it suffices that $\smooth_n = o(1)$, while if $\varepsilon_n \rightarrow 0$ and $\varepsilon_n = \omega(\frac{1}{n})$, then $\smooth_n = o(\varepsilon_n^2)$ is sufficient.
    
    \item {($G = P_n$)} For fixed $\varepsilon$, it is sufficient that $\smooth_n = o(n^{4\theta-3})$, while the condition $\smooth_n = o(\varepsilon_n^2 n^{4\theta-3})$ suffices if $\varepsilon_n \rightarrow 0$ and $\varepsilon_n = \omega(\frac{1}{n^{1-\theta}})$. 
\end{enumerate}
\end{remark}
%
%

\subsection{High probability error bounds}  
We now proceed to derive high probability bounds on the estimation error when $z \in \calC_n$ is generated randomly as in \eqref{eq:noise_mod}.
We begin with a simplification of some of the bounds stated in Proposition \ref{prop:conc_bounds} that we will need in our analysis.
\begin{lemma} \label{lem:lem:simp_conc}
For any given $\lambbar \in [\lambmin, \lambda_1]$, let $U$ denote the $n \times (1 + \abs{\calL_{\lambbar}})$ matrix consisting of $q_j$ for $n - 1 - \abs{\calL_{\lambbar}} \leq j \leq n$. Assuming $z \in \calC_n$ is generated randomly as in \eqref{eq:noise_mod}, denote $\zbar = \expec[z] = \expn{2} h$. Then the following is true.
\begin{enumerate}[label=\upshape(\roman*)]
\item \label{lem:simp_conc_item1} If $\sigma \leq \frac{1}{2\sqrt{2} \pi}$, then $$\prob\left((z-\zbar)^{*} U U^T (z-\zbar) \leq 6190  \sigma^2 \left(\sqrt{(1 + \abs{\calL_{\lambbar}}) \log n} + 1 + \abs{\calL_{\lambbar}} \right) + 4096 \log n \right) \geq 1 - \frac{2}{n^2}.$$

\item \label{lem:simp_conc_item2} If $\frac{72\log n}{\pi\sqrt{n}} \leq \sigma \leq \frac{1}{2\sqrt{2} \pi}$, then $\prob(\norm{z-\zbar}_2^2 \leq 5\pi^2 \sigma^2 n) \geq 1 - \frac{2}{n^2}$.

\item \label{lem:simp_conc_item3} If $\frac{72\log n}{\pi\sqrt{n}} \leq \sigma \leq \frac{1}{2\sqrt{2} \pi}$, then $\prob(\norm{z-h}_2^2 \geq \pi^2 \sigma^2 n) \geq 1 - \frac{2}{n^2}$.
\end{enumerate}
\end{lemma}
\begin{proof}
Recall that for $x \in [0,1]$, we have $x/2 \leq 1 - e^{-x} \leq x$. Therefore if $8\pi^2 \sigma^2 \leq 1$, then 
\begin{equation} \label{eq:useful_bds1}
    1-\expn{8} \in [4 \pi^2 \sigma^2 , 8\pi^2 \sigma^2], \ 1-\expn{4} \in [2 \pi^2 \sigma^2, 4\pi^2 \sigma^2], \ 1-\expn{2} \in [\pi^2 \sigma^2, 2\pi^2 \sigma^2]. 
\end{equation}
\paragraph{Proof of \ref{lem:simp_conc_item1}.} 
Using \eqref{eq:useful_bds1} in Proposition \ref{prop:conc_bounds} \ref{prop:conc_bds_2} with $k = 1+\abs{\calL_{\lambbar}}$ readily yields the statement.
\paragraph{Proof of \ref{lem:simp_conc_item2}.} 
Using \eqref{eq:useful_bds1} in Proposition \ref{prop:conc_bounds} \ref{prop:conc_bds_4}, we obtain w.p at least $1-\frac{2}{n^2}$ the bound
\begin{equation} \label{eq:simp_conc_pt2_tmp}
\norm{z-\zbar}_2^2 \leq 4\pi^2 \sigma^2 n + 3\log n(2+\sqrt{4+72\pi^2 \sigma^2 n})    
\end{equation}
We now simplify the RHS of \eqref{eq:simp_conc_pt2_tmp} by noting that if $\sigma \geq \frac{1}{6\pi \sqrt{n}}$ and $n \geq 2$, then
\begin{equation} \label{eq:simp_conc_pt2_tmp1}
    2+\sqrt{4+72\pi^2 \sigma^2 n} \leq 2 + 12\pi \sigma \sqrt{n} \leq 24 \pi \sigma \sqrt{n}.
\end{equation}
Applying \eqref{eq:simp_conc_pt2_tmp1} in \eqref{eq:simp_conc_pt2_tmp}, we obtain 
\begin{equation*}
    \norm{z-\zbar}_2^2 \leq 4\pi^2 \sigma^2 n + 72\pi\sigma \sqrt{n} \log n \leq 5\pi^2 \sigma^2 n
\end{equation*}
provided $\sigma \geq \frac{72 \log n}{\pi \sqrt{n}}$. 

\paragraph{Proof of \ref{lem:simp_conc_item3}.}
Using \eqref{eq:useful_bds1} in Proposition \ref{prop:conc_bounds} \ref{prop:conc_bds_3}, we obtain w.p at least $1-\frac{2}{n^2}$ the bound
\begin{equation*}  
\norm{z-h}_2^2 \geq 2\pi^2 \sigma^2 n - 3\log n(2+\sqrt{4+72\pi^2 \sigma^2 n}) \geq 2\pi^2 \sigma^2 n - 72\pi\sigma \sqrt{n} \log n     
\end{equation*}
provided $\sigma \geq \frac{1}{6\pi \sqrt{n}}$ and $n \geq 2$ (using \eqref{eq:simp_conc_pt2_tmp1}). The condition $\sigma \geq \frac{72 \log n}{\pi \sqrt{n}}$ then implies $\norm{z-h}_2^2 \geq \pi^2 \sigma^2 n$.
\end{proof}
We will now use Lemma \ref{lem:lem:simp_conc} and Lemma \ref{lem:ucqp_lem_err_1} to obtain a high probability error bound on $\norm{\frac{\gest}{\abs{\gest}} - h}_2$. 
%
%
%
%
\begin{theorem} \label{thm:ucqp_prob_bd}
Let $z \in \calC_n$ be generated as in \eqref{eq:noise_mod}, and assume that $\frac{72\log n}{\pi\sqrt{n}} \leq \sigma \leq \frac{1}{2\sqrt{2} \pi}$.
For any given $\lambbar \in [\lambmin, \lambda_1]$, if $\reg = (\frac{4\pi^2 \sigma^2 n}{\triangle \smooth_n \lambbar^2})^{1/4}$, then with probability at least $1 - \frac{4}{n^2}$, the solution $\gest \in \mathbb{C}^n$ of \eqref{prog:ucqp} satisfies 
\begin{equation} \label{eq:ucqp_prob_err_bd}
    \norm{\frac{\gest}{\abs{\gest}} - h}_2^2 \leq 72\pi \frac{\sigma \sqrt{\triangle \smooth_n n}}{\lambbar}  + 99040 \sigma^2 \left(1+\abs{\calL_{\lambbar}} + \sqrt{(1 + \abs{\calL_{\lambbar}}) \log n} \right) + 65536 \log n .
\end{equation}
\end{theorem}
\begin{proof}
We first simplify the bound in Lemma \ref{lem:ucqp_lem_err_1} with $1+\reg \lambmin \geq 1$, $1 + \reg \lambbar \geq \reg \lambbar$. Denoting $\zbar = \expn{2} h$, this yields
\begin{equation} \label{eq:ucqp_prob_err_bd_tmp_1}
\norm{\mproj{\gest} - h}_2^2 \leq 16 \reg^2 \triangle \smooth_n + 8\expp{4} \sum_{j \in \calL_{\lambbar} \cup \set{n}} \abs{\dotprod{q_j}{z- \zbar}}^2 + 8\frac{\expp{4}}{\reg^2 \lambbar^2} \norm{z-\zbar}_2^2. 
\end{equation}

Applying the bounds in \ref{lem:simp_conc_item1},\ref{lem:simp_conc_item2} of Lemma \ref{lem:lem:simp_conc} in \eqref{eq:ucqp_prob_err_bd_tmp_1}, and observing that $\expp{4} \leq 1 + 8\pi^2\sigma^2 \leq 2$ when $\sigma \leq \frac{1}{2\sqrt{2} \pi}$, we obtain with probability at least $1-\frac{4}{n^2}$ that
\begin{align} \label{eq:ucqp_prob_err_bd_tmp_2}
\norm{\frac{\gest}{\abs{\gest}} - h}_2^2 
\leq 16 \reg^2 \triangle \smooth_n + 99040 \sigma^2 \left(1+\abs{\calL_{\lambbar}} + \sqrt{(1 + \abs{\calL_{\lambbar}}) \log n} \right) + 65536 \log n + \frac{80\pi^2 \sigma^2 n}{\reg^2 \lambbar^2}.
\end{align}
Plugging the stated choice of $\reg$ in \eqref{eq:ucqp_prob_err_bd_tmp_2} leads to  \eqref{eq:ucqp_prob_err_bd}.
\end{proof}
By combining Lemma \ref{lem:lem:simp_conc}\ref{lem:simp_conc_item3} with Theorem \ref{thm:ucqp_prob_bd}, we obtain our main result that establishes conditions under which \eqref{prog:ucqp} provably denoises $z$ with high probability.
\begin{theorem} \label{thm:ucqp_denoise_prob}
Let $z \in \calC_n$ be generated as in \eqref{eq:noise_mod}, and $\varepsilon \in (0,1)$. For any given $\lambbar \in [\lambmin, \lambda_1]$
suppose that 
\begin{equation*}
    \max \set{\frac{69}{\varepsilon \lambbar} \sqrt{\frac{{\triangle \smooth_n}}{n}}, 142  \frac{\log n}{\sqrt{\varepsilon n}}} \leq 
    \sigma \leq \frac{1}{2 \sqrt{2} \pi} \quad \text{ and }  \quad 
    1+\abs{\calL_{\lambbar}} + \sqrt{(1 + \abs{\calL_{\lambbar}}) \log n} \leq \frac{\varepsilon n}{10035}.
\end{equation*}
If $\reg = (\frac{4\pi^2 \sigma^2 n}{\triangle \smooth_n \lambbar^2})^{1/4}$, then with probability at least $1-\frac{6}{n^2}$, the solution $\gest$ of \eqref{prog:ucqp} satisfies
\begin{equation} \label{eq:ucqp_eps_den_bd_prob}
\norm{\frac{\gest}{\abs{\gest}} - h}_2^2 \leq \varepsilon \norm{z - h}_2^2.    
\end{equation}
\end{theorem}
\begin{proof}
Recall from Lemma \ref{lem:lem:simp_conc}\ref{lem:simp_conc_item3}, that if $\frac{72\log n}{\pi\sqrt{n}} \leq \sigma \leq \frac{1}{2\sqrt{2} \pi}$ then $\norm{z - h}_2^2 \geq \pi^2 \sigma^2 n$ holds with high probability. Using \eqref{eq:ucqp_prob_err_bd} from Theorem \ref{thm:ucqp_prob_bd}, we hence note that \eqref{eq:ucqp_eps_den_bd_prob} is ensured if
\begin{equation} \label{ucqp_mainthm_prob_tmp1}
72\pi  \frac{\sigma \sqrt{\triangle \smooth_n n}}{\lambbar}  + 99040 \sigma^2 \left(1+\abs{\calL_{\lambbar}} + \sqrt{(1 + \abs{\calL_{\lambbar}}) \log n} \right) + 65536 \log n \leq \varepsilon \pi^2 \sigma^2 n
\end{equation}
which in turn is ensured provided each LHS term of \eqref{ucqp_mainthm_prob_tmp1} is less than or equal to $\frac{\varepsilon \pi^2 \sigma^2 n}{3}$. Combining the resulting three conditions with the requirement $\frac{72\log n}{\pi\sqrt{n}} \leq \sigma \leq \frac{1}{2\sqrt{2} \pi}$, one can check that the stated conditions in the Theorem suffice.
\end{proof}
As in Corollary \ref{cor:ucqp_den_expec}, we can translate the conditions of Theorem \ref{thm:ucqp_denoise_prob} for the special cases $G = K_n, S_n$ or $P_n$. The proof is similar to that of Corollary \ref{cor:ucqp_den_expec} and hence omitted.
\begin{corollary} \label{cor:ucqp_den_prob}
 Let $z \in \calC_n$ be generated as in \eqref{eq:noise_mod}, and $\varepsilon \in (0,1)$. 
 \begin{enumerate}[label=\upshape(\roman*)]
     \item\label{cor:ucqp_den_prob_eps_Kn} {($G = K_n$)} Suppose $\frac{n}{\sqrt{\log n}} \gtrsim \frac{1}{\varepsilon}$ and $\max\set{\frac{1}{n \varepsilon} \sqrt{\smooth_n}, \frac{\log n}{\sqrt{\varepsilon n}}  } \lesssim \sigma \lesssim 1$.  If $\reg \asymp (\frac{\sigma^2}{n^2 \smooth_n })^{1/4}$, then the solution $\gest$ of \eqref{prog:ucqp} satisfies \eqref{eq:ucqp_eps_den_bd_prob} w.h.p.    
     
     \item\label{cor:ucqp_den_prob_eps_Sn} {($G = S_n$)} Suppose $\frac{n}{\sqrt{\log n}} \gtrsim \frac{1}{\varepsilon}$ and $\max\set{\frac{1}{\varepsilon} \sqrt{\smooth_n}, \frac{\log n}{\sqrt{\varepsilon n}} } \lesssim \sigma \lesssim 1$.  If $\reg \asymp (\frac{\sigma^2}{\smooth_n })^{1/4}$, then the solution $\gest$ of \eqref{prog:ucqp} satisfies \eqref{eq:ucqp_eps_den_bd_prob} w.h.p.
     
     \item\label{cor:ucqp_den_prob eps_Pn} {($G = P_n$)} For a given $\theta \in [0,1)$, suppose $n^{\theta} + \sqrt{n^{\theta} \log n} \lesssim \varepsilon n$ and $\max\set{\frac{n^{\frac{3 -4\theta}{2}}}{\varepsilon} \sqrt{\smooth_n}, \frac{\log n}{\sqrt{\varepsilon n}} } \lesssim \sigma \lesssim 1$.  If $\reg \asymp (\frac{\sigma^2 n^{5-4\theta}}{\smooth_n })^{1/4}$, then the solution $\gest$ of \eqref{prog:ucqp} satisfies \eqref{eq:ucqp_eps_den_bd_prob} w.h.p. 
 \end{enumerate}
\end{corollary}
\begin{remark} \label{rem:cor_ex_graph_prob_ucqp}
As in Remark \ref{rem:cor_ex_graph_expec_ucqp}, we can deduce conditions on the smoothness term $\smooth_n$ which lead to a non-vacuous regime for $\sigma$ of the form $o(1) \leq \sigma \lesssim 1$. These are detailed below when $n \rightarrow \infty$.
\begin{enumerate}
    \item {($G = K_n$)} If $\varepsilon$ is fixed then $\smooth_n = o(n^2)$ suffices while if $\varepsilon = \varepsilon_n \rightarrow 0$ and $\varepsilon_n = \omega(\frac{\log n}{\sqrt{n}})$, then $\smooth_n = o(\varepsilon_n^2 n^2)$ is sufficient. 
    
    \item {($G = S_n$)} If $\varepsilon$ is fixed then $\smooth_n = o(1)$ suffices while if $\varepsilon_n \rightarrow 0$ and $\varepsilon_n = \omega(\frac{\log n}{\sqrt{n}})$, then $\smooth_n = o(\varepsilon_n^2)$ is sufficient.
    
    \item {($G = P_n$)} For fixed $\varepsilon$, it is sufficient that $\smooth_n = o(n^{4\theta-3})$. On the other hand, if $\varepsilon_n \rightarrow 0$  and $\varepsilon_n = \omega(\max\{\frac{\log n}{\sqrt{n}}, \frac{n^{\theta} + \sqrt{n^{\theta} \log n}}{n}\})$ then the condition $\smooth_n = o(\varepsilon_n^2 n^{4\theta-3})$ suffices. 
\end{enumerate}
\end{remark}
%
%
%
\paragraph{Denoising modulo samples of a function.} We conclude this section by instantiating the aforementioned bounds to the setting described in Section \ref{subsec:prob_setup}. Recall that here, we obtain noisy modulo $1$ samples of a smooth function $f: [0,1] \rightarrow \mathbb{R}$ with $G = P_n$. This leads to the following Corollary of Theorems \ref{thm:ucqp_prob_bd} and \ref{thm:ucqp_denoise_prob}.
\begin{corollary} \label{cor:ucqp_den_func_mod}
Consider the example from Section \ref{subsec:prob_setup} where we obtain noisy modulo 1 samples of a $M$-Lipschitz function $f:[0,1] \rightarrow \matR$. If $\reg \asymp \left(\frac{\sigma^2 n^{10/3}}{M^2} \right)^{1/4}$ then the following is true for the solution $\gest$ of \eqref{prog:ucqp}.
\begin{enumerate}
    \item If $\frac{\log n}{\sqrt{n}} \lesssim \sigma \lesssim 1$ and $n \gtrsim 1$ then w.h.p,
    \begin{equation*}
       \norm{\frac{\gest}{\abs{\gest}} - h}_2^2 \lesssim (\sigma M  + \sigma^2) n^{2/3} + \log n.
    \end{equation*}

    \item For $\varepsilon \in (0,1)$, 
    if $n \gtrsim (1/\varepsilon)^3$ and  $\max\set{\frac{M}{\varepsilon n^{1/3}}, \frac{\log n}{ \sqrt{\varepsilon n}} } \lesssim \sigma \lesssim 1$, then $\gest$ satisfies \eqref{eq:ucqp_eps_den_bd_prob} w.h.p.
\end{enumerate}
\end{corollary}
\begin{proof}
Starting with \eqref{eq:func_quad_var_Bn}, we have 
\begin{equation} \label{eq:func_quad_var_Bn_1}
  \sum_{i=1}^{n-1} \abs{h_i - h_{i+1}}^2 = 4\pi^2 \sum_{i=1}^{n-1} \abs{f(x_i) - f(x_{i+1})}^2 \leq \frac{4\pi^2 M^2}{n-1} \leq \frac{8 \pi^2 M^2}{n} \ (= \smooth_n)
\end{equation}
where the penultimate inequality uses the Lipschitz continuity of $f$, and the last inequality uses $n-1 \geq n/2$. 

For the first part, recall from Corollary \ref{cor:ucqp_den_expec} \ref{cor:ucqp_den_eps_Pn} that $\lambbar \asymp \frac{1}{n^{2(1-\theta)}}$ for $\theta \in [0,1)$, which implies $\abs{\calL_{\lambbar}} \lesssim n^{\theta}$. Applying this along with \eqref{eq:func_quad_var_Bn_1} to Theorem \ref{thm:ucqp_prob_bd}, we obtain the bound 
\begin{equation*}
      \norm{\frac{\gest}{\abs{\gest}} - h}_2^2 \lesssim \sigma M n^{2(1-\theta)} + \sigma^2 (n^{\theta} + \sqrt{n^{\theta} \log n}) + \log n.
\end{equation*}
Notice that we need $1/2 < \theta < 1$ to get a non-trivial bound. The choice $\theta = 2/3$ ``balances'' the exponents of $n$, and if moreover $n \gtrsim 1$, then we obtain the statement of the first part.

The second part follows easily by plugging \eqref{eq:func_quad_var_Bn_1} along with $\theta = 2/3$ in Corollary \ref{cor:ucqp_den_prob}\ref{cor:ucqp_den_prob eps_Pn}.
\end{proof}
\begin{remark} \label{rem:ucqp_lw_bd_subopt}
We can see that the estimation error bound for \eqref{prog:ucqp} is significantly better than the one in \eqref{eq:cmt_err_bd} derived by Cucuringu and Tyagi \cite{CMT18_long} for \eqref{prog:trs}.  For a constant $\sigma$, our $\ell_2$ error bound scales as $O(n^{1/3})$ for large enough $n$, however this is not optimal. In a parallel work \cite{fanuel20} with the present paper, it is shown that for the complex valued function $h(x) = \exp(\iota 2\pi f(x))$, for any given $x \in [0,1]$, one can obtain an estimate $\est{h}(x)$ such that $\expec[\est{h}(x) - h(x)]^2 = O(n^{-2/3})$ which is also consistent with the pointwise minimax rate for estimating a Lipschitz function\footnote{Note that if $f:[0,1] \rightarrow \matR$ is $M$-Lipschitz, then it implies that $h(\cdot)$ is $2\pi M$ Lipschitz since $\abs{h(x) - h(x')} \leq 2\pi \abs{f(x) - f(x')} \leq 2\pi M \abs{x-x'}$ for all $x,x' \in [0,1]$.} \cite{nemirovski2000topics}. This suggests that for the example from Section \ref{subsec:prob_setup}, the best $\ell_2$ error bound we can hope for is $O(n^{1/6})$. It will be interesting to see whether the analysis for \eqref{prog:ucqp} can be improved to an extent such that instantiating the result for $G = P_n$ achieves this optimal bound, we discuss this further in Section \ref{subsec:fut_work}; see also Remark \ref{rem:sadhanala_bounds} in Section \ref{subsec:learn_smooth_rel}. 
\end{remark}

\section{Error bounds for \eqref{prog:trs}} \label{sec:l2_TRS}
We now proceed to derive a $\ell_{2}$ error bound for the solution $\gest$ of \eqref{prog:trs}. Following the notation in \cite{CMT18_long}, we can represent any $x \in \mathbb{C}^n$ via $\bar{x} \in \mathbb{R}^{2n}$, where $\bar{x} = [\real(x)^T \  \imag(x)^T]^T$. Moreover, consider the matrix 
\begin{eqnarray*} 
H = \begin{pmatrix}
  \reg L \quad & 0 \\ 
  0 \quad & \reg L  
 \end{pmatrix}  =  \reg \begin{pmatrix}
   1 & 0 \\ 
   0 & 1   
 \end{pmatrix} \otimes L
 \;\;  \in  \matR^{2n \times 2n}
\end{eqnarray*} 
where $\otimes$ denotes the Kronecker product. Then it is easy to check that \eqref{prog:trs} is equivalent to 
\begin{align} 
\min_{\gbar \in \mathbb{R}^{2n}: \norm{\gbar}_2^2 = n}  \gbar^T H \gbar - 2 \gbar^T \zbar \label{prog:trs_real} \tag{$\text{TRS}\matR$}
\end{align}
where $\gest$ is a solution of \eqref{prog:trs} iff $\gbarest = [\real(\gest)^T \ \imag(\gest)^T]^T$ is a solution of \eqref{prog:trs_real}.

The following Lemma from \cite{CMT18_long}, which in turn is a direct consequence of \cite[Lemma 2.4, 2.8]{Sorensen82} (see also \cite[Lemma 1]{Hager01}),  characterizes any solution $\gbarest$ of \eqref{prog:trs_real}.
%
\begin{lemma}[\cite{CMT18_long}] \label{lemma:trs_real_sol_charac}
For any given $\zbar \in \mathbb{R}^{2n}$, $\gbarest$ is a solution to \eqref{prog:trs_real} iff $\norm{\gbarest}_2^2 = n$ and $\exists \mu^{\star}$ such that
(a) $2 H + \mu^{\star} I \succeq 0$ and (b) $(2 H + \mu^{\star} I) \gbarest = 2\zbar$. Moreover, if $2 H + \mu^{\star} I \succ 0$, then the solution is unique.
\end{lemma}
Due to the equivalence of \eqref{prog:trs} and \eqref{prog:trs_real} as discussed above, this readily leads to the following characterization for any solution $\gest$ of \eqref{prog:trs}.
%
%
\begin{lemma} \label{lemma:trs_sol_charac}
For any given $z \in \mathbb{C}^n$, $\gest$ is a solution to \eqref{prog:trs} iff $\norm{\gest}_2^2 = n$ and $\exists \mu^{\star}$ such that
(a) $2 \reg L + \mu^{\star} I \succeq 0$ and (b) $(2 \reg L + \mu^{\star} I) \gest = 2z$. Moreover, if $2 \reg L + \mu^{\star} I \succ 0$, then the solution is unique.
\end{lemma}
Note that the above Lemma's do not require any sphere constraint on $\zbar, z$. We now use Lemma \ref{lemma:trs_sol_charac} to derive the following crucial Lemma which states upper and lower bounds on $\mu^{\star}$. The notation $\calN(L)$ is used to denote the null space of $L$, which is the span of $q_n$ since $G$ is connected by assumption.
\begin{lemma} \label{lemma:useful_trs_sol}
For any $z \in \calC_n$ satisfying $z \not\perp$ $\calN(L)$, we have that $\gest = 2(2 \reg L + \mu^{\star} I)^{-1} z$ is the unique solution of \eqref{prog:trs} with $\mu^{\star} \in (0,2]$. Additionally, if $\lambtil$ is an eigenvalue of $L$ satisfying 
\begin{equation*}
     \frac{1}{\sqrt{n}} \left(\sum_{j:\lambda_j \leq \lambtil} \abs{\dotprod{z}{q_j}}^2\right)^{1/2} > \reg \lambtil
\end{equation*}
then it holds that 
\begin{equation*}
\frac{2}{\sqrt{n}} \left(\sum_{j:\lambda_j \leq \lambtil} \abs{\dotprod{z}{q_j}}^2\right)^{1/2} - 2 \reg \lambtil \leq \mu^{\star} \leq 2.
\end{equation*}
\end{lemma}
\begin{proof}
Let us denote
\begin{align} \label{eq:phimu}
    \phi(\mu) := \norm{2(2\reg L +  \mu I)^{-1} z}_2^2 = 4\sum_{j=1}^{n} \frac{\abs{\dotprod{z}{q_j}}^2}{(2 \reg\lambda_j + \mu)^2}.
\end{align}
If $z \not\perp$ $\calN(L)$ then $0$ is a pole of $\phi$. Hence there exists a unique $\mu^{\star} \in (0,\infty)$ such that $\gest = 2(2 \reg L + \mu^{\star} I)^{-1} z$ is a (unique) solution of \eqref{prog:trs} as it satisfies the conditions in Lemma \ref{lemma:trs_sol_charac}. Since $n = \norm{\gest}_2^2  \leq \frac{4n}{(\mu^{\star})^2}$, we obtain $\mu^{\star} \leq 2$. To obtain the lower bound on $\mu^{\star}$, note that  \eqref{eq:phimu}  implies 
\begin{equation*}
    \norm{\gest}_2^2 = n = \norm{2(2\reg L +  \mu^{\star} I)^{-1} z}_2^2 \geq \frac{ \sum_{\lambda_j \leq \lambtil} \abs{\dotprod{z}{q_j}}^2}{(\reg \lambtil + \frac{\mu^{\star}}{2})^2}
\end{equation*}
which leads to the stated lower bound on $\mu^{\star}$.
\end{proof}
This is an important Lemma since localizing the value of $\mu^{\star}$ will be key for controlling the $\ell_2$ error bound for \eqref{prog:trs}. Observe that $\lambtil = 0$ always satisfies the conditions of the Lemma since $\abs{\dotprod{z}{q_n}} > 0$.

Next, we bound the the error between $\frac{\gest}{\abs{\gest}}$ and $h$ \emph{for any} $z \in \calC_n$ such that $z \not\perp \calN(L)$.
\begin{lemma} \label{lem:trs_err_bd_det}
For any $z \in \calC_n$ such that $z \not\perp$ $\calN(L)$, and $\lambbar \in [\lambmin, \lambda_1]$, the (unique) solution $\gest = 2(2 \reg L + \mu^{\star} I)^{-1} z$ of \eqref{prog:trs} satisfies the bound
\begin{align*}
 \norm{\frac{\gest}{\abs{\gest}} - h}_2^2 \leq \frac{32}{(\mu^{\star})^2} (E_1 + E_2) + 8 \left(\frac{2}{\mu^{\star}} - 1 \right)^2 \left(   \abs{\dotprod{h}{q_n}}^2 + \frac{\sum_{j \in \calL_{\lambbar}} \abs{\dotprod{h}{q_j}}^2}{(1+\reg \lambmin)^2} 
    + \frac{\smooth_n}{\lambbar (1+\reg \lambbar)^2}  \right)
\end{align*}
with $E_1, E_2$ as defined in Lemma \ref{lem:ucqp_lem_err_1}, and $\mu^{\star} \in (0,2]$.
\end{lemma}
\begin{proof}
The proof is along the lines of Lemma \ref{lem:ucqp_lem_err_1} with only few technical differences. Firstly, we can write
\begin{equation*}
    \expp{2} \gest - h = \underbrace{\frac{2}{\mu^{\star}} \left(I + \frac{2\reg}{\mu^{\star}} L \right)^{-1} (\expp{2} z - h)}_{e_1} + \underbrace{\frac{2}{\mu^{\star}} \left(I + \frac{2\reg}{\mu^{\star}} L \right)^{-1}h - h}_{e_2}
\end{equation*}
which in conjunction with Proposition \ref{prop:entry_proj} leads to the bound
\begin{equation*}
    \norm{\mproj{\gest} - h}_2^2 = 
    \norm{\mproj{\expp{2} \gest} - h}_2^2 \leq 8(\norm{e_1}_2^2 + \norm{e_2}_2^2).
\end{equation*}
Proceeding identically to the proof of Lemma \ref{lem:ucqp_lem_err_1}, it is easy to verify that $\norm{e_1}_2^2 \leq \frac{4}{(\mu^{\star})^2} E_1$. 
%
%
In order to bound $\norm{e_2}_2^2$, we begin by expanding $e_2$ as 
\begin{equation*}
e_2 = \frac{2}{\mu^{\star}} \left(I + \frac{2\reg}{\mu^{\star}} L \right)^{-1}h - h 
= \left(\frac{2}{\mu^{\star}} - 1 \right) \dotprod{q_n}{h} q_n  
+  \sum_{j=1}^{n-1} \left(\frac{\frac{2}{\mu^{\star}}}{1 + \frac{2}{\mu^{\star}}\reg \lambda_j} - 1 \right) \dotprod{q_j}{h} q_j.    
\end{equation*}
Then using the orthonormality of $q_j$'s,  we can bound $\norm{e_2}_2^2$ as follows.
\begin{align*}
    \norm{e_2}_2^2 
    &\leq \left(\frac{2}{\mu^{\star}} - 1 \right)^2 \abs{\dotprod{h}{q_n}}^2 +  \sum_{j=1}^{n-1} \left(\frac{(\frac{2}{\mu^{\star}} - 1)^2 + \frac{4\reg^2 \lambda_j^2}{(\mu^{\star})^2}}{(1 + \frac{2}{\mu^{\star}}\reg \lambda_j)^2}   \right) \abs{\dotprod{h}{q_j}}^2 \quad (\text{using } (a-b)^2 \leq a^2 + b^2 \text{ for } a,b \geq 0 ) \\
    &\leq \left(\frac{2}{\mu^{\star}} - 1 \right)^2 \abs{\dotprod{h}{q_n}}^2 +  \sum_{j =1}^{n-1}  \left(\frac{(\frac{2}{\mu^{\star}} - 1)^2 + \frac{4\reg^2 \lambda_j^2}{(\mu^{\star})^2}}{(1 + \reg \lambda_j)^2}   \right) \abs{\dotprod{h}{q_j}}^2 \quad (\text{using } \mu^{\star} \leq 2) \\
    %
    %
    &\leq \left(\frac{2}{\mu^{\star}} - 1 \right)^2 \left( \abs{\dotprod{h}{q_n}}^2 + \frac{ \sum_{j \in \calL_{\lambbar}} \abs{\dotprod{h}{q_j}}^2}{(1+\reg \lambmin)^2} + \frac{ \sum_{j \in \calH_{\lambbar}} \abs{\dotprod{h}{q_j}}^2}{(1+\reg \lambbar)^2} \right) + \frac{4\reg^2}{(\mu^{\star})^2 (1+\reg \lambmin)^2} \sum_{j=1}^{n-1} \lambda_j^2 \abs{\dotprod{h}{q_j}}^2 \\
    &\leq \left(\frac{2}{\mu^{\star}} - 1 \right)^2 \left( \abs{\dotprod{h}{q_n}}^2 + \frac{ \sum_{j \in \calL_{\lambbar}} \abs{\dotprod{h}{q_j}}^2}{(1+\reg \lambmin)^2} 
    + \frac{ \smooth_n}{\lambbar (1+\reg \lambbar)^2} \right) + \frac{4}{(\mu^{\star})^2} E_2,
\end{align*}
where in the last inequality, we used \eqref{eq:smooth_res1},\eqref{eq:smooth_res2}.
\end{proof}
When $z \in \calC_n$ is generated as in \eqref{eq:noise_mod}, the following Lemma presents a (high probability) lower bound  on $\mu^{\star}$ provided $\sigma$ is small and $h$ is sufficiently smooth.
%
\begin{lemma} \label{lem:lowbd_mu_prob}
Let $z \in \calC_n$ be generated as in \eqref{eq:noise_mod}, then the solution of $\eqref{prog:trs}$ is unique. Moreover, suppose that for any given $k \in [n-1]$ s.t $\lambda_{n-k+1} < \lambda_{n-k}$, the following holds.
\begin{equation} \label{eq:conds_lowbd_mu}
    \frac{\smooth_n}{\lambda_{n-k}} \leq \frac{n}{12}, \quad \sigma^2 \leq \frac{1}{48\pi^2}, \quad \frac{24760\log n}{\sqrt{n}} \leq \frac{1}{12} \quad \text{ and }  \reg \lambda_{n-k+1} \leq \frac{1}{4}.
\end{equation}
%
%
Then with probability at least $1 - \frac{4}{n^2}$, we have that 
\begin{equation*}
 \frac{\mu^{\star}}{2} \geq 1 - \left(\frac{\smooth_n}{n \lambda_{n-k}} + 4\pi^2\sigma^2 + \frac{24760 \log n}{\sqrt{n}} + \reg \lambda_{n-k+1} \right).    
\end{equation*}
\end{lemma}
\begin{proof}
We will lower bound the lower bound estimate of $\mu^{\star}$ from Lemma \ref{lemma:useful_trs_sol} using Proposition \ref{prop:conc_bounds}\ref{prop:conc_bds_1}. Note that $z \in \calC_n$ satisfies $z \not\perp \calN(L)$ a.s. Set $\lambtil = \lambda_{n-k+1}$ in Lemma \ref{lemma:useful_trs_sol}, and let $U$ denote the $n \times k$ matrix consisting of $q_j$'s for $n-k+1 \leq j \leq n$. 

Let us first simplify the statement of Proposition \ref{prop:conc_bounds}\ref{prop:conc_bds_1} when $\sigma^2 \leq \frac{1}{8\pi^2}$. Recall from the proof of Lemma \ref{lem:lem:simp_conc} that this implies $1-\expn{8} \in [4 \pi^2 \sigma^2 , 8\pi^2 \sigma^2]$. Then the (magnitude of the) RHS of the bound in Proposition \ref{prop:conc_bounds}\ref{prop:conc_bds_1} can be upper bounded as
\begin{align}
    & 4096 \log n +256 \sqrt{6} \pi^2 \sigma^2 \sqrt{k \log n}     
     + 11 \log n \left(\norm{U U^T h}_{\infty} + 2\sqrt{2} \pi \sigma \underbrace{\norm{U U^T h}_2}_{\leq \sqrt{n}} \right) \nonumber \\
&\leq 6190 (\log n + \sigma^2 \sqrt{k \log n} + \log n(\norm{UU^T h}_{\infty} + \sigma \sqrt{n})) \nonumber \\
&\leq 6190 \log n (1 + \norm{U U^T h}_{\infty} + 2\sigma \sqrt{n}) \label{eq:imp_up_bound}
\end{align}
where the last inequality uses $\sigma^2 \leq \sigma$ and $k \leq n$. 

Plugging \eqref{eq:imp_up_bound} in Proposition \ref{prop:conc_bounds} \ref{prop:conc_bds_1}, we conclude that with probability at least $1-\frac{4}{n^2}$, 
\begin{align*}
    \sum_{j:\lambda_j \leq \lambda_{n-k+1}} \abs{\dotprod{z}{q_j}}^2 
    &\geq 2\pi^2 \sigma^2 k + (1-4\pi^2\sigma^2) h^* UU^T h - \underbrace{6190 \log n (1 + \norm{UU^T h}_{\infty} + 2\sigma \sqrt{n})}_{\leq 24760 \sqrt{n} \log n} \\
    &\geq (1-4\pi^2\sigma^2) h^* UU^T h - 24760 \sqrt{n} \log n \\
    &\geq (1-4\pi^2\sigma^2) \left(n - \frac{\smooth_n}{\lambda_{n-k}} \right) - 24760 \sqrt{n} \log n \quad \text{ ($h^*UU^T h \geq n - \frac{\smooth_n}{\lambda_{n-k}}$, see \eqref{eq:smooth_res1})} \\
    &\geq n - \frac{\smooth_n}{\lambda_{n-k}} - 4\pi^2\sigma^2 n - 24760 \sqrt{n} \log n  \\ 
    &= n \left(1 - \frac{\smooth_n}{n \lambda_{n-k}} - 4\pi^2\sigma^2 - \frac{24760 \log n}{\sqrt{n}} \right). 
\end{align*}
Using \eqref{eq:conds_lowbd_mu}, we have $\frac{\smooth_n}{n \lambda_{n-k}} + 4\pi^2\sigma^2 + \frac{24760 \log n}{\sqrt{n}} \leq \frac{1}{4}$. This leads to the bound 
\begin{equation*}
    \left(\frac{1}{n} \sum_{j:\lambda_j \leq \lambda_{n-k+1}} \abs{\dotprod{z}{q_j}}^2 \right)^{1/2} \geq 1 - \frac{\smooth_n}{n \lambda_{n-k}} - 4\pi^2\sigma^2 - \frac{24760 \log n}{\sqrt{n}}
\end{equation*}
which in conjunction with Lemma \ref{lemma:useful_trs_sol} leads to the stated bound on $\mu^{\star}$. In particular, the conditions in \eqref{eq:conds_lowbd_mu} ensure $\mu^{\star} \geq 1$.
\end{proof}
Using Lemmas \ref{lem:trs_err_bd_det} and \ref{lem:lowbd_mu_prob}, we arrive at the following (high probability) bound on the error $\norm{\frac{\gest}{\abs{\gest}} - h}_2^2$ for the solution $\gest$ of $\eqref{prog:trs}$.
\begin{theorem} \label{thm:trs_err_bd_prob}
Let $z \in \calC_n$ be generated as in \eqref{eq:noise_mod}, then the solution $\gest$ of \eqref{prog:trs} is unique. For any given $k \in [n-1]$ s.t $\lambda_{n-k+1} < \lambda_{n-k}$, and any $\lambbar \in [\lambmin, \lambda_1]$ with the choice $\reg = (\frac{4\pi^2 \sigma^2 n}{\triangle \smooth_n \lambbar^2})^{1/4}$, suppose that the following conditions are satisfied.
\begin{enumerate} [label=\upshape(\roman*)]
    \item\label{lem:trs_err_bd_cond1} $\smooth_n \leq  \min\set{\frac{n \lambda_{n-k}}{12}, \frac{n \lambbar}{2}}$,  and
    
    \item\label{lem:trs_err_bd_cond2} $286 \left(\frac{\log n}{\sqrt{n}} \right)^{1/2} \leq \sigma \leq \min\set{\frac{1}{4\sqrt{3} \pi},\frac{\lambbar}{16 \lambda_{n-k+1}^2} \sqrt{\frac{\triangle \smooth_n}{4\pi^2 n}}}$. 
\end{enumerate}
Then with probability at least $1 - \frac{8}{n^2}$, the solution $\gest \in \mathbb{C}^n$ of \eqref{prog:trs} satisfies
\begin{align} \label{eq:trs_prob_err_bd}
    \norm{\frac{\gest}{\abs{\gest}} - h}_2^2 
    &\leq C_1 \frac{\sigma}{\lambbar} \left( \sqrt{\triangle \smooth_n n} +  \frac{n^{3/2} \lambda_{n-k+1}^2}{\sqrt{\triangle \smooth_n}}  \right)  
    + C_2 \sigma^2(1 + \abs{\calL_{\lambbar}} + \sqrt{(1 + \abs{\calL_{\lambbar}}) \log n}) \\ \nonumber 
    &+ C_3 \sigma^4 n + C_4 \log n + C_5 \frac{\smooth_n^2}{n \lambda_{n-k}^2}
\end{align}
where $C_1 = 288\pi$, $C_2 = 396160$, $C_3 = 230400$, $C_4 = 262144$, $C_5 = 144$.
\end{theorem}
\begin{proof}
We simply combine Lemmas \ref{lem:trs_err_bd_det} and \ref{lem:lowbd_mu_prob}. To this end, recall that the error bound in Theorem \ref{thm:ucqp_prob_bd} is a bound on the term $8(E_1 + E_2)$. This means that if $\frac{72\log n}{\pi \sqrt{n}} \leq \sigma \leq \frac{1}{2\sqrt{2} \pi}$, then for the stated choice of $\reg$, we have with probability at least $1-\frac{4}{n^2}$,
\begin{equation} \label{eq:first_err_term_trs_hp}
    32(E_1 + E_2) \leq 288 \pi \frac{\sigma}{\lambbar} \sqrt{\triangle \smooth_n n} + 396160 \sigma^2(1 + \abs{\calL_{\lambbar}} + (1 + \abs{\calL_{\lambbar}})\sqrt{\log n}) +  262144 \log n. 
\end{equation}
The conditions in Lemma \ref{lem:lowbd_mu_prob} imply in particular that $\mu^{\star} \geq 1$, hence  \eqref{eq:first_err_term_trs_hp} is a bound on $\frac{32(E_1 + E_2)}{\mu^{\star}}$. This takes care of the first error term in the bound in Lemma \ref{lem:trs_err_bd_det}.

In order to bound the second term therein, observe that if $\smooth_n \leq \frac{n \lambbar}{2}$, then 
\begin{equation*}
     \abs{\dotprod{h}{q_n}}^2 + \frac{\sum_{j \in \calL_{\lambbar}} \abs{\dotprod{h}{q_j}}^2}{(1+\reg \lambmin)^2} 
    + \frac{\smooth_n}{\lambbar (1+\reg \lambbar)^2} \leq n + \frac{n}{2} = \frac{3n}{2}.
\end{equation*}
Also, condition \ref{lem:trs_err_bd_cond2} for $\sigma$ implies $24760 \frac{\log n}{\sqrt{n}} \leq  \frac{3\sigma^2}{\pi^2} (\leq \frac{1}{12})$. Note that the requirement $\sigma \geq 286(\frac{\log n}{\sqrt{n}})^{1/2}$ is stricter than $\sigma \geq \frac{72 \log n}{\pi \sqrt{n}}$.  Given these observations, we can bound the second term in the bound of Lemma \ref{lem:trs_err_bd_det} as follows.
\begin{align}
    & 8 \left(\frac{2}{\mu^{\star}} - 1 \right)^2 \left(   \abs{\dotprod{h}{q_n}}^2 + \frac{\sum_{j \in \calL_{\lambbar}} \abs{\dotprod{h}{q_j}}^2}{(1+\reg \lambmin)^2} 
    + \frac{\smooth_n}{\lambbar (1+\reg \lambbar)^2}  \right) \nonumber \\
    &\leq 48n \left(\frac{\smooth_n}{n \lambda_{n-k}} + 4\pi^2\sigma^2 + \frac{24760 \log n}{\sqrt{n}} + \reg \lambda_{n-k+1} \right)^2 \nonumber \\
    &\leq    
     48n \left(\frac{\smooth_n}{n \lambda_{n-k}} + 4\pi^2\sigma^2 + \frac{3\sigma^2}{\pi^2} + \reg \lambda_{n-k+1} \right)^2 \nonumber \\
    &\leq 48n \left(\frac{\smooth_n}{n \lambda_{n-k}} + 40\sigma^2 +  \reg \lambda_{n-k+1} \right)^2 \nonumber \\
    &\leq 144 n \left( \frac{\smooth_n^2}{n^2 \lambda_{n-k}^2} + 1600 \sigma^4  
    + \left(\frac{4\pi^2 \sigma^2 n}{\triangle \smooth_n \lambbar^2} \right)^{1/2} \lambda_{n-k+1}^2 \right). \label{eq:sec_err_term_trs_hp}
\end{align}
Plugging \eqref{eq:first_err_term_trs_hp} and \eqref{eq:sec_err_term_trs_hp} in Lemma \ref{lem:trs_err_bd_det} then readily yields the stated bound in the Theorem. 
\end{proof}
The following Corollary provides a simplification of Theorem \ref{thm:trs_err_bd_prob} and is directly obtained by considering $k=1$ since $\lambda_n < \lambda_{n-1} = \lambmin$. 
%
%
%
\begin{corollary} \label{cor:trs_high_prob_simp_err}
Let $z \in \calC_n$ be generated as in \eqref{eq:noise_mod}, then the solution $\gest$ of \eqref{prog:trs} is unique. For given $\lambbar \in [\lambmin, \lambda_1]$ with the choice $\reg = (\frac{4\pi^2 \sigma^2 n}{\triangle \smooth_n \lambbar^2})^{1/4}$, suppose that 
\begin{equation*}
    \smooth_n \leq  \frac{n \lambmin}{12} \text{  and  }  286 \left(\frac{\log n}{\sqrt{n}} \right)^{1/2} \leq \sigma \leq \frac{1}{4\sqrt{3} \pi}.
\end{equation*}
Then with probability at least $1 - \frac{8}{n^2}$, the solution $\gest \in \mathbb{C}^n$ of \eqref{prog:trs} satisfies
\begin{align*} 
    \norm{\frac{\gest}{\abs{\gest}} - h}_2^2 
    \leq C_1 \frac{\sigma}{\lambbar} \sqrt{\triangle \smooth_n n} +    
    + C_2 \sigma^2 \left(1 + \abs{\calL_{\lambbar}} + \sqrt{(1 + \abs{\calL_{\lambbar}}) \log n} \right)   
    + C_3 \sigma^4 n + C_4 \log n + C_5 \frac{\smooth_n^2}{n \lambmin^2}
\end{align*}
where the constants $C_1,\dots,C_5$ are as in Theorem \ref{thm:trs_err_bd_prob}.
\end{corollary}
We are now in a position to derive conditions under which \eqref{prog:trs} provably denoises $z$ with high probability. We begin with the following Theorem which provides these conditions in their full generality.
%
%
\begin{theorem} \label{thm:trs_prob_denoise}
Let $z \in \calC_n$ be generated as in \eqref{eq:noise_mod}, then the solution $\gest$ of \eqref{prog:trs} is unique. With constants $C_1, \dots, C_5$ as in Theorem \ref{thm:trs_err_bd_prob}, for any  $\varepsilon \in (0,1)$,  given $k \in [n-1]$ s.t $\lambda_{n-k+1} < \lambda_{n-k}$  and $\lambbar \in [\lambmin, \lambda_1]$ with the choice $\reg = (\frac{4\pi^2 \sigma^2 n}{\triangle \smooth_n \lambbar^2})^{1/4}$, suppose that the following conditions are satisfied.

\begin{enumerate}[label=\upshape(\roman*)]
    \item\label{lem:trs_denoise_cond1} $\smooth_n \leq  \min\set{\frac{n \lambda_{n-k}}{12}, \frac{n \lambbar}{2}}$,  and $1 + \abs{\calL_{\lambbar}} + \sqrt{(1 + \abs{\calL_{\lambbar}}) \log n} \leq \frac{\pi^2}{5 C_2}\varepsilon n$.
    
     \item\label{eq:trs_denoise_cond2}   $\sigma \leq
    \min\set{\frac{\pi \sqrt{\varepsilon}}{\sqrt{5 C_3}},\frac{\lambbar}{16 \lambda_{n-k+1}^2} \sqrt{\frac{\triangle \smooth_n}{4\pi^2 n}}}$ and 
    \begin{align*}
        \sigma \geq \max \left\{286 \left(\frac{\log n}{\sqrt{n}} \right)^{1/2}, \frac{\sqrt{5 C_5}}{\pi} \left(\frac{\smooth_n}{n \lambda_{n-k} \sqrt{\varepsilon}} \right), \sqrt{\left(\frac{5C_4}{\varepsilon \pi^2} \right) \frac{\log n}{n}}, \frac{5C_1}{\pi^2 \varepsilon \lambbar} \left(\sqrt{\frac{\triangle \smooth_n}{n}} + \lambda_{n-k+1}^2 \sqrt{\frac{n}{\triangle \smooth_n}} \right) \right\}. 
    \end{align*}
\end{enumerate}    
Then with probability at least $1 - \frac{10}{n^2}$, the solution $\gest \in \mathbb{C}^n$ of \eqref{prog:trs} satisfies 
\begin{equation} \label{eq:trs_eps_den_bd_prob}
    \norm{\frac{\gest}{\abs{\gest}} - h}_2^2 \leq \varepsilon \norm{z - h}_2^2.   
\end{equation}
\end{theorem}
\begin{proof}
Recall from Lemma \ref{lem:lem:simp_conc} \ref{lem:simp_conc_item3}, that $\norm{z-h}_2^2 \geq \pi^2 \sigma^2 n$ w.p at least $1-\frac{2}{n^2}$. Conditioning on the intersection of this event and the event in Theorem \ref{thm:trs_err_bd_prob}, it suffices to ensure that 
the bound in \eqref{eq:trs_prob_err_bd} is less than or equal to $\pi^2\sigma^2 n \varepsilon$. This in turn is ensured provided each term in the RHS of \eqref{eq:trs_prob_err_bd} is less than or equal to $\varepsilon \frac{\pi^2 \sigma^2 n}{5}$. 
Combining the resulting conditions with those in Theorem \ref{thm:trs_err_bd_prob} yields the statement of the Theorem.
\end{proof}
The following simplification of Theorem \ref{thm:trs_prob_denoise} is obtained for $k=1$, as was done in Corollary \ref{cor:trs_high_prob_simp_err}.
%
%
%
\begin{corollary} \label{cor:trs_high_prob_simp_den}
Let $z \in \calC_n$ be generated as in \eqref{eq:noise_mod}, then the solution $\gest$ of \eqref{prog:trs} is unique. With constants $C_1,\dots, C_5$ as in Theorem \ref{thm:trs_err_bd_prob}, for any  $\varepsilon \in (0,1)$ and $\lambbar \in [\lambmin, \lambda_1]$ with the choice $\reg = (\frac{4\pi^2 \sigma^2 n}{\triangle \smooth_n \lambbar^2})^{1/4}$, suppose that the following conditions are satisfied.

\begin{enumerate}[label=\upshape(\roman*)]
    \item\label{cor:trs_denoise_simp_cond1} $\smooth_n \leq  \frac{n \lambmin}{12}$ and $1 + \abs{\calL_{\lambbar}} + \sqrt{(1 + \abs{\calL_{\lambbar}}) \log n} \leq \frac{\pi^2}{5 C_2}\varepsilon n$.
    
     \item\label{cor:trs_denoise_simp_cond2}   
    \begin{align*}
        \max\set{286 \left(\frac{\log n}{\sqrt{n}} \right)^{1/2}, \frac{\sqrt{5 C_5}}{\pi} \left(\frac{\smooth_n}{n \lambmin \sqrt{\varepsilon}} \right), \sqrt{\left(\frac{5C_4}{\varepsilon \pi^2} \right) \frac{\log n}{n}}, 
        \frac{5 C_1}{\pi^2 \varepsilon \lambbar}  \sqrt{\frac{\triangle \smooth_n}{n}}} \leq \sigma \leq  \frac{\pi \sqrt{\varepsilon}}{\sqrt{5 C_3}}.
    \end{align*}
\end{enumerate}    
Then with probability at least $1 - \frac{10}{n^2}$, the solution $\gest \in \mathbb{C}^n$ of \eqref{prog:trs} satisfies \eqref{eq:trs_eps_den_bd_prob}.
\end{corollary}
Finally, as done previously for \eqref{prog:ucqp}, it will be instructive to translate Theorem \ref{thm:trs_prob_denoise} for the special cases $G = K_n$, $S_n$ or $P_n$. This is stated below using the simplified version in Corollary \ref{cor:trs_high_prob_simp_den}. 
\begin{corollary} \label{cor:trs_den_prob_spec_graphs}
Let $z \in \calC_n$ be generated as in \eqref{eq:noise_mod} and $\varepsilon \in (0,1)$.

 \begin{enumerate}[label=\upshape(\roman*)]
     \item\label{cor:trs_den_prob_eps_Kn} {($G = K_n$)} Suppose $\frac{n}{\sqrt{\log n}} \gtrsim \frac{1}{\varepsilon}$, $\smooth_n \lesssim n^2$ and $\max\set{\frac{\sqrt{\smooth_n}}{n \varepsilon} , \frac{\smooth_n}{n^2 \sqrt{\varepsilon}} , (\frac{\log n}{\sqrt{n}})^{1/2}, (\frac{\log n}{\varepsilon n})^{1/2} } \lesssim \sigma \lesssim \sqrt{\varepsilon}$.  If $\reg \asymp (\frac{\sigma^2}{n^2 \smooth_n })^{1/4}$, then the (unique) solution $\gest$ of \eqref{prog:trs} satisfies \eqref{eq:trs_eps_den_bd_prob} w.h.p.    
     
     \item\label{cor:trs_den_prob_eps_Sn} {($G = S_n$)} Suppose $\frac{n}{\sqrt{\log n}} \gtrsim \frac{1}{\varepsilon}$, $\smooth_n \lesssim n$ and 
     $\max\set{\frac{\sqrt{\smooth_n}}{\varepsilon} , \frac{\smooth_n}{n \sqrt{\varepsilon}},  (\frac{\log n}{\sqrt{n}})^{1/2}, (\frac{\log n}{\varepsilon n})^{1/2}} \lesssim \sigma \lesssim \sqrt{\varepsilon}$.  If $\reg \asymp (\frac{\sigma^2}{\smooth_n })^{1/4}$, then the (unique) solution $\gest$ of \eqref{prog:trs} satisfies \eqref{eq:trs_eps_den_bd_prob} w.h.p.
     
     \item\label{cor:trs_den_prob eps_Pn} {($G = P_n$)} For a given $\theta \in [0,1)$, suppose $n^{\theta} + \sqrt{n^{\theta} \log n} \lesssim \varepsilon n$, $\smooth_n \lesssim \frac{1}{n}$ and  
     $$\max\set{\frac{n^{\frac{3 -4\theta}{2}}}{\varepsilon} \sqrt{\smooth_n}, \frac{n \smooth_n}{\sqrt{\varepsilon}},  \left(\frac{\log n}{\sqrt{n}} \right)^{1/2}, \left(\frac{\log n}{\varepsilon n} \right)^{1/2}} \lesssim \sigma \lesssim \sqrt{\varepsilon}.$$  
     If $\reg \asymp (\frac{\sigma^2 n^{5-4\theta}}{\smooth_n })^{1/4}$, then the (unique) solution $\gest$ of \eqref{prog:trs} satisfies \eqref{eq:trs_eps_den_bd_prob} w.h.p. 
\end{enumerate}
\end{corollary}
\begin{proof}
Use Corollary \ref{cor:trs_high_prob_simp_den} with $\lambmin, \lambbar$ as in Corollary \ref{cor:ucqp_den_expec}.
\end{proof}
For $K_n$, note that only $k = 1$ meets the requirement of Theorem \ref{thm:trs_prob_denoise} since $\lambda_{n-1} = \cdots = \lambda_{1}$. For $S_n$, the only other possibility (apart from $k=1$) is to choose $k = n-1$, since $\lambda_{2} = 1 < \lambda_1 = n$. But this choice of $k$ leads to a vacuous noise regime due to the appearance of the term $\sqrt{\smooth_n} + \frac{1}{\sqrt{\smooth_n}}$ as a lower bound on $\sigma$. 

\begin{remark} \label{rem:trs_exgraph_large_n}
Similarly to Remark \ref{rem:cor_ex_graph_prob_ucqp} for \eqref{prog:ucqp}, we can deduce conditions on the smoothness term $\smooth_n$ which -- when $n \rightarrow \infty$ -- lead to a non-vacuous regime for $\sigma$ of the form $o(1) \leq \sigma \lesssim \sqrt{\varepsilon}$. Here, we will only treat the case where $\varepsilon$ is fixed.
\begin{enumerate}
    \item {($G = K_n$)} $\smooth_n = o(n^2)$ suffices. 
    
    \item {($G = S_n$)} $\smooth_n = o(1)$ suffices.
    
    \item {($G = P_n$)} $\smooth_n = o(1/n)$ suffices.
\end{enumerate}
\end{remark}
\paragraph{Denoising modulo samples of a function.}
When $G = P_n$, the requirement $\smooth_n = o(1/n)$ is far from satisfactory and suggests that Corollary \ref{cor:trs_high_prob_simp_den} is weak when applied to a path graph. Indeed, when we obtain noisy modulo $1$ samples of a $M$-Lipschitz function $f:[0,1] \rightarrow \matR$, then we have $\smooth_n \asymp  \frac{M^2}{n}$ as seen in \eqref{eq:func_quad_var_Bn_1}. Unfortunately, the condition on $\sigma$ in Corollary \ref{cor:trs_den_prob_spec_graphs}\ref{cor:trs_den_prob eps_Pn} becomes vacuous when $\smooth_n \asymp 1/n$. Interestingly, we can handle this smoothness regime by making use of Theorem \ref{thm:trs_prob_denoise} with a careful choice of $k$. This is made possible by the fact that the spectrum of the Laplacian of $P_n$ satisfies the condition  $\lambda_{n-k+1} < \lambda_{n-k}$ for each $k \in [n-1]$. Consequently, we obtain the following Corollary of Theorems \ref{thm:trs_err_bd_prob} and \ref{thm:trs_prob_denoise}; its proof is deferred to Appendix \ref{appsec:proof_Pn_res_bet_conds}.
\begin{corollary} \label{cor:Pn_res_bet_conds}
Consider the example from Section \ref{subsec:prob_setup} where we obtain noisy modulo 1 samples of a $M$-Lipschitz function $f:[0,1] \rightarrow \matR$. If $\reg \asymp \left(\frac{\sigma^2 n^{10/3}}{M^2} \right)^{1/4}$ then the following is true for the solution $\gest$ of \eqref{prog:trs}.
\begin{enumerate}
    \item If $n \gtrsim \max \set{1,M^2}$ and $(\frac{\log n}{\sqrt{n}})^{1/2} \lesssim \sigma \lesssim \min \set{1, n^{1/3} M}$ then w.h.p,
    \begin{equation*}
       \norm{\frac{\gest}{\abs{\gest}} - h}_2^2 \lesssim \left(\sigma \left(M + \frac{1}{M} \right) + \sigma^2 \right) n^{2/3} + \sigma^4 n + \log n + \frac{M^4}{n}.
    \end{equation*}

    \item For $\varepsilon \in (0,1)$, 
    if $n \gtrsim \max \set{(1/\varepsilon)^3, M^2}$ and 
\begin{equation*}
    \max\set{\frac{1}{\varepsilon n^{1/3}} \left(M + \frac{1}{M} \right) , \frac{M^2}{n \sqrt{\varepsilon}},  \left(\frac{\log n}{\sqrt{n}} \right)^{1/2}, \left(\frac{\log n}{\varepsilon n} \right)^{1/2}} \lesssim \sigma \lesssim \min\set{\sqrt{\varepsilon} , n^{1/3} M},
\end{equation*}
    then $\gest$ satisfies \eqref{eq:ucqp_eps_den_bd_prob} w.h.p.
\end{enumerate}
\end{corollary}
The error bound in Corollary \ref{cor:Pn_res_bet_conds} is visibly worse than that in Corollary \ref{cor:ucqp_den_func_mod} due to the appearance of an additional $\sigma^4 n$ term.  For $n$ large enough and $\varepsilon \in (0,1)$ fixed, Corollary \ref{cor:Pn_res_bet_conds} asserts that \eqref{prog:trs} succeeds in denoising in the noise regime $(\frac{\log n}{\sqrt{n}})^{1/2} \lesssim \sigma \lesssim \sqrt{\varepsilon}$. The corresponding noise regime for \eqref{prog:ucqp} is the relatively weaker requirement $\frac{M}{\varepsilon n^{1/3}} \lesssim \sigma \lesssim 1$, as seen from Corollary \ref{cor:ucqp_den_func_mod}.

\section{Simulations} \label{sec:sims}
\rev{We now provide simulation results on some synthetic examples. For concreteness, we consider the following functions.}
\begin{enumerate}
    \item \rev{$f_1(x) = 3x \cos^2(2 \pi x) - \sin^2(2\pi x) + 0.7$}, 
    
    \item \rev{$f_2(x) = \sin (2\pi x)$.}
\end{enumerate}
\rev{The function $f_1 (x) \bmod 1$ is relatively more complicated than $f_2 (x) \bmod 1$ as the former has more number number of ``folds'' or ``jumps'' than the latter. Following the notation and setup in the example described in Section \ref{subsec:prob_setup}, we sample the functions on a uniform grid in $[0,1]$ (containing $n = 500$ points) according to the Gaussian noise model in \eqref{eq:mod1_noisy_samp_func}. }
\begin{figure}[!ht]
\centering
\subcaptionbox[]{$f_1$ (hard)}[ 0.4\textwidth ]
{\includegraphics[width=0.4\textwidth]{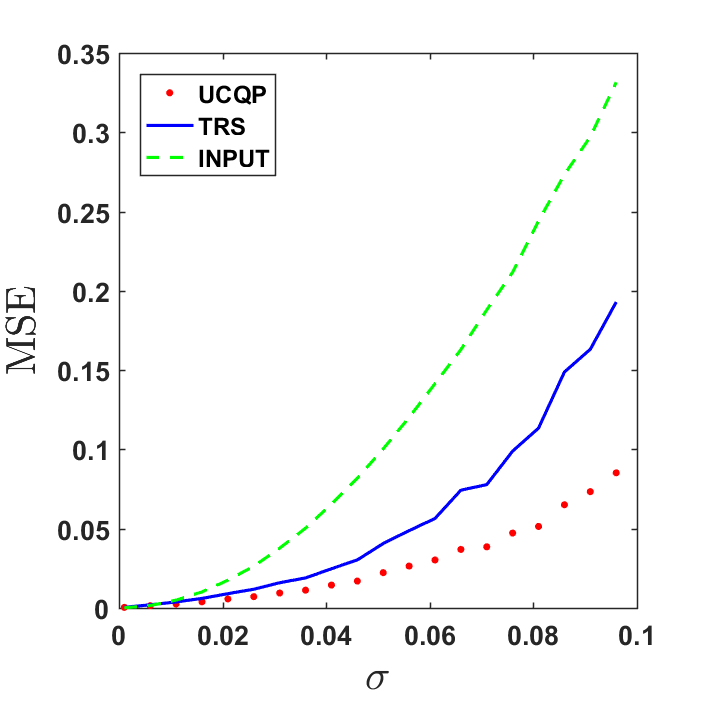} } 
\subcaptionbox[]{$f_2$ (easy)}[ 0.4\textwidth ]
{\includegraphics[width=0.4\textwidth]{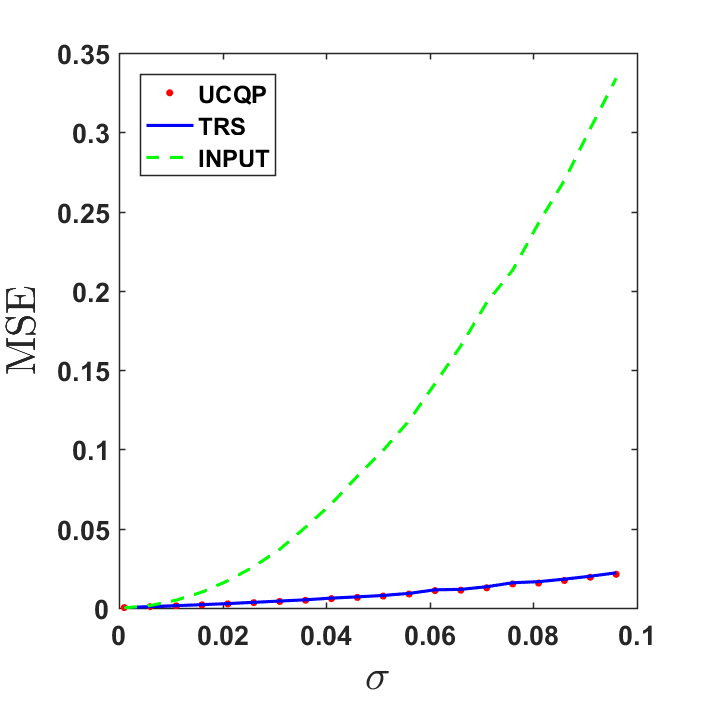} } 
%
 %

\subcaptionbox[]{$f_1$ (hard)}[ 0.4\textwidth ]
{\includegraphics[width=0.4\textwidth]{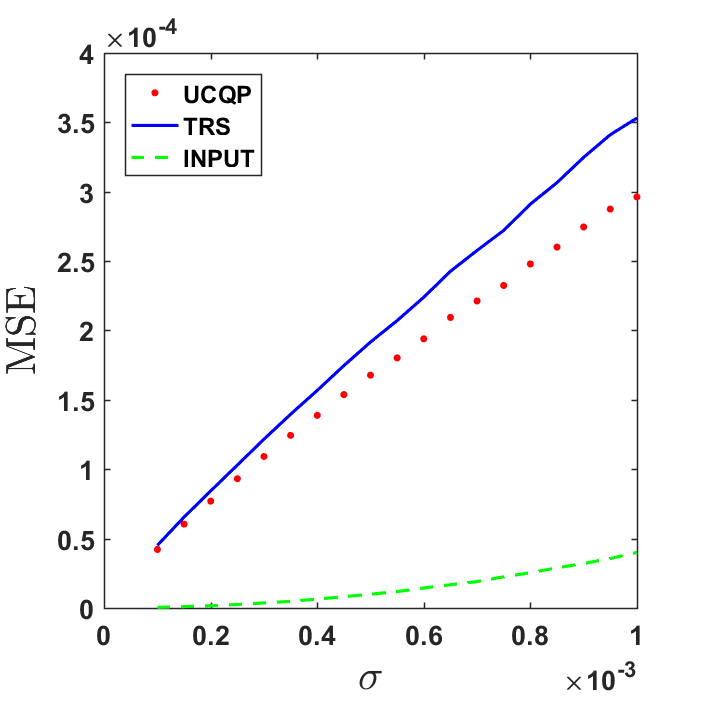} } 
\subcaptionbox[]{$f_2$ (easy)}[ 0.4\textwidth ]
{\includegraphics[width=0.4\textwidth]{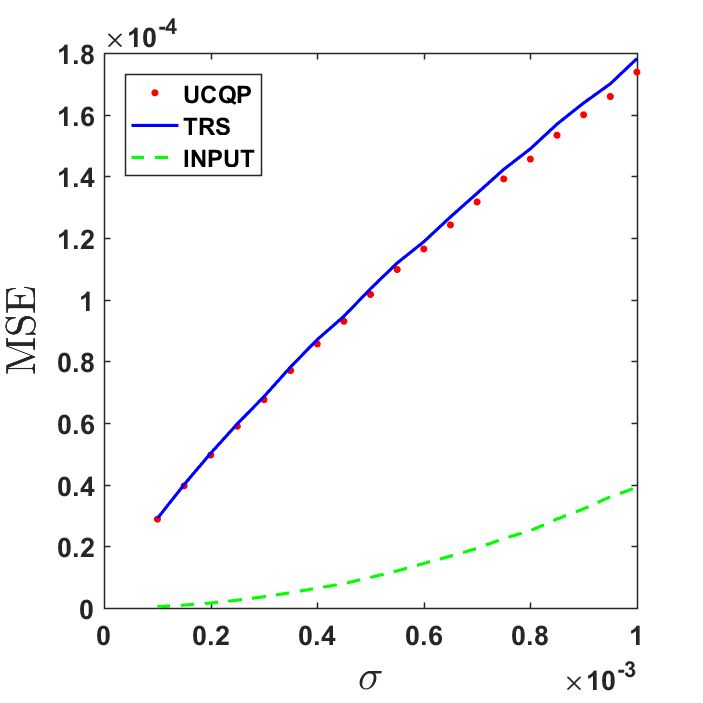} } 
%
%
\captionsetup{width=0.98\linewidth}
\caption[Short Caption]{MSE vs $\sigma$ for the estimators and the input. Top two figures are for $\sigma$ ranging from $10^{-3}$ to $0.096$. Bottom two figures are for smaller values of $\sigma$ ranging from $10^{-4}$ to $10^{-3}$. Throughout, we set $\reg =  (\sigma^2 n^{10/3})^{1/4}$ as specified in Corollaries \ref{cor:ucqp_den_func_mod} and \ref{cor:Pn_res_bet_conds} for \eqref{prog:ucqp} and \eqref{prog:trs} respectively. Results are averaged over $30$ trials. }
\label{fig:plots_n_500_var_sigma}
\end{figure}
\rev{Taking $G = P_n$, our aim is to demonstrate the behaviour of the mean square error (MSE) of the estimators, for different noise levels $\sigma$. In particular, we are interested in checking whether   $\norm{\frac{\gest}{\abs{\gest}} - h}_2^2$ is less than $\norm{z - h}_2^2$ (MSE of the input) with $\gest$ a solution of \eqref{prog:ucqp} or \eqref{prog:trs}.}

\rev{The results are illustrated in Figure \ref{fig:plots_n_500_var_sigma} for $\reg$ as specified in Corollaries \ref{cor:ucqp_den_func_mod} and \ref{cor:Pn_res_bet_conds}. The top two plots therein show the MSE values for $\sigma$ ranging from $10^{-3}$ to $0.096$. As $\sigma$ increases, the denoising performance of the estimators becomes more apparent. Interestingly, \eqref{prog:trs} performs worse than \eqref{prog:ucqp} for the hard input ($f_1$), but they both exhibit similar performance for the easier case ($f_2$). When $\sigma$ is very small (between $10^{-4}$ and $10^{-3}$), we can see from the bottom two plots in Figure  \ref{fig:plots_n_500_var_sigma} that the MSE of the estimators have a slightly larger value than that of the input. Hence for very low values of $\sigma$, the denoising performance is not seen. This is also consistent with the statements of Corollaries \ref{cor:ucqp_den_func_mod} and \ref{cor:Pn_res_bet_conds} which require $\sigma \gtrsim o(1)$ for guaranteed denoising of the input.} 

\rev{It is important to keep in mind that the value of the regularizer $\reg$ that we used is not optimal since it does not yield the optimal dependency of the error bounds in terms of $n$ for this specific setup (as noted in Remark \ref{rem:ucqp_lw_bd_subopt}).} \secrev{For the optimal choice of $\reg$, we would expect that the denoising performance is exhibited for very low values of $\sigma$ as well. To illustrate this, we repeat the previous experiment (with $n = 500$ fixed) but this time with $\reg = 400*\sigma$. Note that in Figure \ref{fig:plots_n_500_var_sigma}, we had chosen $\reg = (500)^{5/6} \sigma^{1/2} \approx 177.5* \sigma^{1/2} $. For this new choice of $\reg$ (see Figure \ref{fig:plots_n_500_var_sigma_better_reg}), we see that denoising also occurs for low values of $\sigma$ (in the range $10^{-4}$ and $10^{-3}$) with similar performance for \eqref{prog:ucqp} and \eqref{prog:trs} in this noise regime. For larger values of $\sigma$ (in the range $10^{-3}$ to $0.096$), the top two plots in Figure \ref{fig:plots_n_500_var_sigma_better_reg} are similar to those of Figure \ref{fig:plots_n_500_var_sigma}.}

\begin{figure}[!ht]
\centering
\subcaptionbox[]{$f_1$ (hard)}[ 0.4\textwidth ]
{\includegraphics[width=0.4\textwidth]{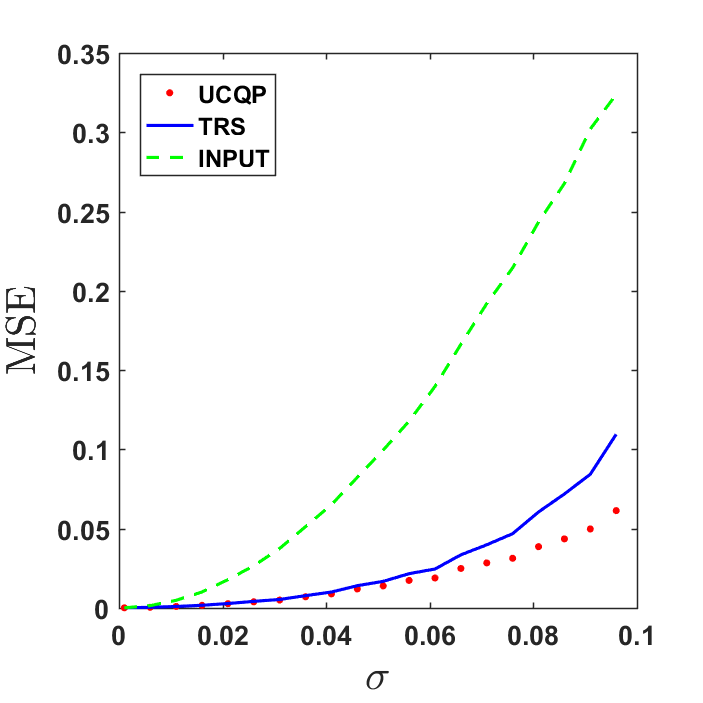} } 
\subcaptionbox[]{$f_2$ (easy)}[ 0.4\textwidth ]
{\includegraphics[width=0.4\textwidth]{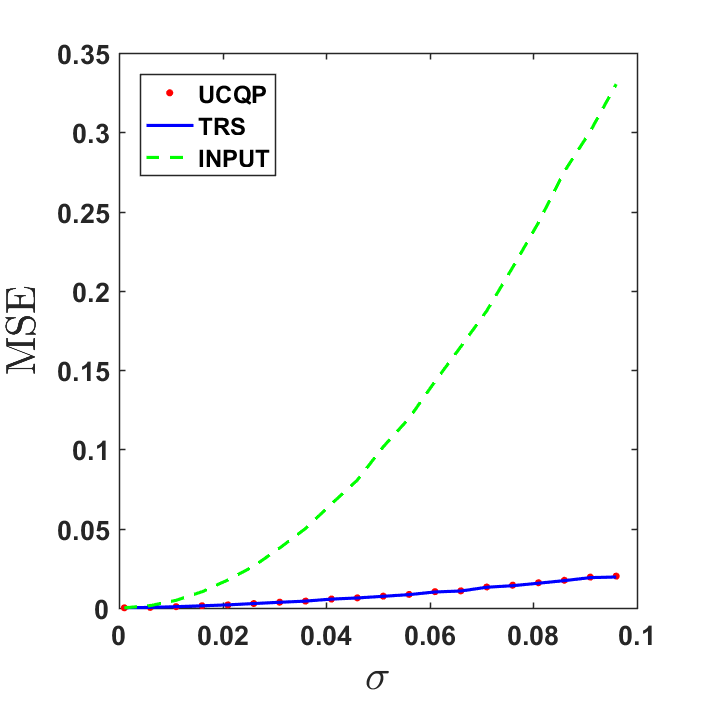} } 
%
 %

\subcaptionbox[]{$f_1$ (hard)}[ 0.4\textwidth ]
{\includegraphics[width=0.4\textwidth]{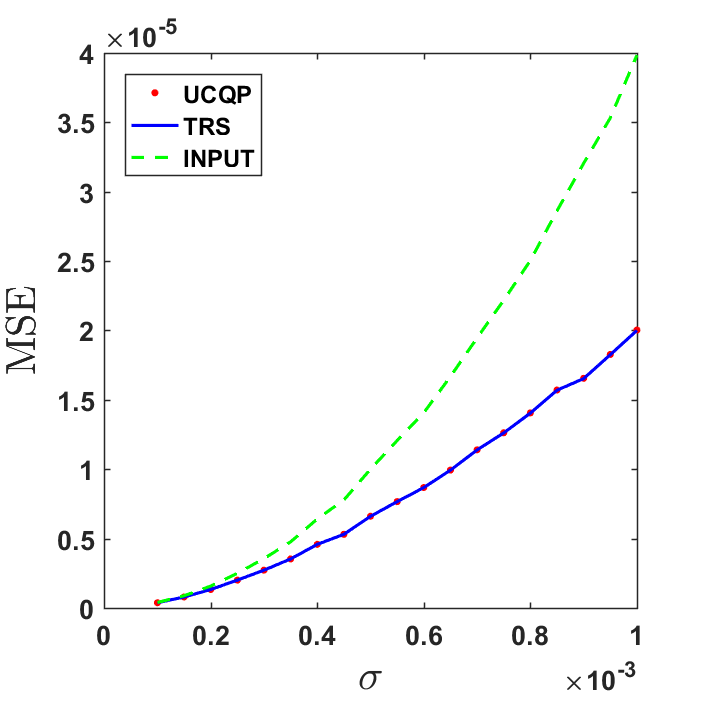} } 
\subcaptionbox[]{$f_2$ (easy)}[ 0.4\textwidth ]
{\includegraphics[width=0.4\textwidth]{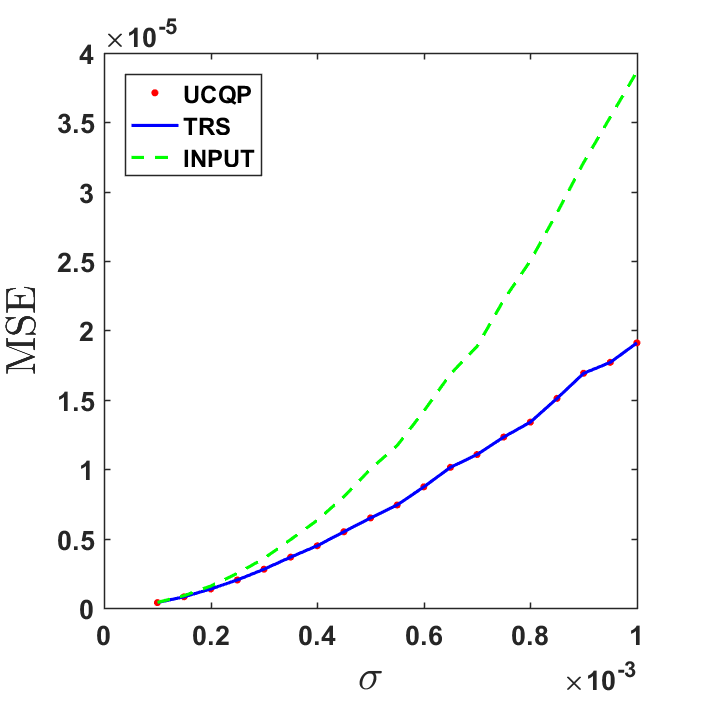} } 
%
%
\captionsetup{width=0.98\linewidth}
\caption[Short Caption]{MSE vs $\sigma$ for the estimators and the input. Top two figures are for $\sigma$ ranging from $10^{-3}$ to $0.096$. Bottom two figures are for smaller values of $\sigma$ ranging from $10^{-4}$ to $10^{-3}$. Throughout, we set $\reg =  400*\sigma$ . Results are averaged over $30$ trials. }
\label{fig:plots_n_500_var_sigma_better_reg}
\end{figure}
%
%

\section{Discussion} \label{sec:discussion}
We now discuss in detail some related work and conclude with directions for future work.
\subsection{Related work} \label{subsec:existing_work} 
There exist numerous methods for this problem in the phase unwrapping community, most of which are for the setting $d = 2$ (since this case has the most number of applications). Such methods can be broadly classified as belonging to the class of (a) least squares based approaches (e.g., \cite{huang2012path, Rivera04half}), (b) branch cut methods (e.g., \cite{Chavez02, Prati90}), or (c) network flow methods (e.g., \cite{Towers91, takeda96}). While we refer the reader to \cite{CMT18_long} for a more detailed discussion of these methods (as well as other related approaches from phase unwrapping), we remark that most of these approaches are based on heuristics with no theoretical performance guarantees. 

A recent line of work for this problem has led to the development of new methods with provable performance guarantees. Bhandari et al. \cite{bhandari17} considered equispaced sampling of a univariate bandlimited function (with spectrum in $[-\pi,\pi]$) and showed in the noiseless setting that if the sampling width is less than or equal to $\frac{1}{2\pi e}$, then the samples of $g$ (and consequently the function $g$ itself) can be recovered exactly. 
This work was extended by the same set of authors to other settings such as in \cite{sparse_unlim18}, where $f$ is assumed to be the convolution of a low pass filter and a sum of $k$ Diracs, and in \cite{sparse_unlimsin18} where $f$ is considered to be a sum of $k$ sinusoids. Then, given $n$ equispaced (with step size $T$) noiseless modulo measurements of $f$, Bhandari et al. \cite{sparse_unlim18,sparse_unlimsin18} show that $f$ can be recovered exactly provided $n$ is large enough (roughly speaking, $n \gtrsim k$), and $T \leq \frac{1}{2\pi e}$. In a follow up work, Rudresh et al. \cite{rudresh_wavelet} considered the setting where $f$ is a univariate Lipschitz function, and proposed a method based on first applying a wavelet filter to the (equispaced) modulo samples, followed by a LASSO based procedure for recovering $f$. They showed that if $f$ is a polynomial of degree $p$, then it can be recovered exactly from its noiseless modulo samples provided the sampling width is $\lesssim \frac{\zeta}{L p}$, where $L$ is the Lipschitz constant\footnote{It should of course depend on $p$, but this was not stated explicitly in \cite{rudresh_wavelet}} of $f$. The authors do not provide any theoretical guarantees in the presence of noise, however demonstrate via simulation results that their method is more robust to noise than that of Bhandari et al. \cite{bhandari17}. 

While the aforementioned results \cite{bhandari17,sparse_unlim18,sparse_unlimsin18,rudresh_wavelet} are for the nonparametric setting and with $f$ being univariate, the setting where $f$ is a $d$-variate linear function was considered by Shah and Hegde \cite{shah2018signal}. Assuming $f$ to be sparse, exact recovery guarantees were provided (for the noiseless setting) in the regime $n \ll d$ for an alternating minimization based algorithm. Musa et al. \cite{MusaJG18} also consider $f$ to be a sparse linear function, but assume it to be generated from a Bernoulli-Gaussian distribution. They propose a generalized approximate message passing algorithm for recovering $f$, but without any theoretical analysis.   

In the work of Cucuringu and Tyagi \cite{CMT18_long}, the authors also proposed a semi-definite programming (SDP) relaxation of \eqref{prog:qcqp} and also considered solving \eqref{prog:qcqp} using methods for optimization over manifolds. These approaches were shown to perform well via simulations, but without any theoretical analysis. 

In a parallel work with the present paper, Fanuel and Tyagi \cite{fanuel20} derived a two-stage algorithm for unwrapping noisy modulo $1$ samples of a Lipschitz function $f:[0,1]^d \rightarrow \matR$, for the model \eqref{eq:noise_mod} with $\eta_i \sim \calN(0,\sigma^2)$ i.i.d. In the first stage, they represent the noisy data $(y_i)_{i=1}^n$ on the product manifold $\calC_n$ (as in Section \ref{subsec:den_smooth_mod_intro}), and consider estimating the ground signal $h \in \calC_n$ via a $k$NN ($k$ nearest neighbor) procedure. Assuming that the $x_i$'s form a uniform grid on $[0,1]^d$, they show that the ensuing estimate $\est{h} \in \calC_n$ satisfies (w.h.p) the $\ell_{\infty}$ error rate $||\est{h} - h ||_{\infty} \lesssim (\frac{\log n}{n})^{\frac{1}{d+2}}$ when $n$ is large enough. This error rate is then translated to a uniform error bound (w.r.t the wrap-around distance) between $f(x_i) \mod 1$ and the mod $1$ estimate obtained from $\est{h}_i$, for each $i$. Then in the second stage, they present a sequential unwrapping procedure for unwrapping these denoised $\mod 1$ samples and show that the estimates $\tilde{f}(x_i)$ of $f(x_i)$ satisfy (up to a global integer shift) the bound $\abs{\tilde{f}(x_i) - f(x_i)} \lesssim (\frac{\log n}{n})^{\frac{1}{d+2}}$ for each $i$.

Fanuel and Tyagi \cite{fanuel20} also studied the problem of identifying conditions under which the SDP formulation of Cucuringu and Tyagi \cite{CMT18_long} is a \emph{tight} relaxation of \eqref{prog:qcqp}. This is done under a general graph based setup as in the present paper. Without any statistical assumptions on the noisy data $z \in \calC_n$, their result states that if $||z-h||_{\infty} \lesssim 1$ and $\reg \Delta \lesssim 1$, then the SDP relaxation of \eqref{prog:qcqp} is tight, and consequently leads to the global solution of \eqref{prog:qcqp}. As discussed in \cite{fanuel20}, the derived conditions are stricter than what one would expect, and so there is still room for improvement in this regard. 

\subsection{Learning smooth functions on graphs} \label{subsec:learn_smooth_rel}
Estimating a smooth function $\theta^{\star} \in \matR^n$ on a graph $G = ([n],E)$ is a well studied problem in signal processing (\secrev{e.g., \cite{shuman13}}) and statistics (\secrev{e.g., \cite{belkin04,sadhanala16,kirichenko2017,kirichenko2018}}). The statistical model is typically assumed to be
\begin{equation} \label{eq:mod_smooth_fn_graph}
    y = \theta^{\star} + \eta, \quad \eta \sim \calN(0,\sigma^2 I)
\end{equation}
with the smoothness of $\theta^{\star}$ measured by $(\theta^{\star})^{\top} L \theta^{\star}$ which is assumed to be small. 
A common approach for estimating $\theta^\star$ is via the so-called Tikhonov regularization where we aim to solve
\begin{equation} \label{eq:tikh_reg}
    \min_{\theta \in \matR^n} \norm{\theta - y}_2^2 + \reg \theta^\top L \theta.
\end{equation}
While \eqref{eq:tikh_reg} is the same as \eqref{prog:ucqp}, the model in \eqref{eq:mod_smooth_fn_graph} is notably different from \eqref{eq:noise_mod}. 

Let us review some important theoretical results pertaining to \eqref{eq:tikh_reg}. Belkin et al \cite{belkin04} considered the semi-supervised learning problem of predicting the values of $\theta^\star$ on the vertices of a partially labelled graph. They use the notion of algorithmic stability to derive generalization error bounds for \eqref{eq:tikh_reg} which in particular depend on the Fiedler eigenvalue of $L$. Sadhanala et al. \cite{sadhanala16} consider the problem of estimating $\theta^\star$ under the assumption that $G$ is a $d$-dimensional regular grid. The smoothness assumption\footnote{translated to our notation in the present paper} on $\theta^\star$ is that $\theta^\star \in \mathcal{S}(\smooth_n)$ where 
$$\mathcal{S}(\smooth_n) = \set{\theta \in \matR^n : \theta^\top L \theta \leq \smooth_n}.$$ They establish a lower bound on the minimax risk \cite[Theorem 5]{sadhanala16} for the class $\mathcal{S}(\smooth_n)$,
\begin{equation*}
    \inf_{\substack{\est{\theta}}}  \sup_{\substack{\theta^\star \in \mathcal{S}(\smooth_n)}} \frac{1}{n}\expec[\|\est{\theta} - \theta^{\star}\|_2^2] \geq \frac{c}{n} \min\set{(n\sigma^2)^{\frac{2}{d+2}} (\smooth_n)^{\frac{d}{d+2}}, n\sigma^2, n^{2/d} \smooth_n} + \frac{\sigma^2}{n}
\end{equation*}
with $c > 0$ a universal constant. Moreover, they show for $d=1,2,3$ that the solution $\est{\theta}$ of \eqref{eq:tikh_reg} is minimax optimal since it satisfies 
\begin{equation} \label{eq:upbd_sandhanala}
     \sup_{\substack{\theta^\star \in \mathcal{S}(\smooth_n)}} \frac{1}{n}\expec[\|\est{\theta} - \theta^{\star}\|_2^2] \leq \frac{\tilde{c}}{n} \min\set{(n\sigma^2)^{\frac{2}{d+2}} (\smooth_n)^{\frac{d}{d+2}}, n\sigma^2, n^{2/d} \smooth_n} + \frac{\tilde{c}\sigma^2}{n}
\end{equation}
for a universal constant $\tilde{c} > 0$, when $n$ is large enough and $\reg \asymp (\frac{n}{\smooth_n})^{\frac{2}{d+2}}$. It is also mentioned in \cite[Remark 5]{sadhanala16} that for $d = 4$, \eqref{eq:tikh_reg} is nearly minimax optimal due to an extra $\log$ factor in the rate. No rates are provided for $d \geq 5$, but they conjecture that it is not optimal for this range of $d$. It is also shown that the Laplacian eigenmap estimator\footnote{Here, we project $y$ onto the subspace spanned by the $k$ smallest eigenvectors of $L$, for a suitably chosen $k$.} achieves the aforementioned upper bound for all $d$, and hence is minimax optimal. 
\begin{remark} \label{rem:sadhanala_bounds}
The crucial step in establishing \eqref{eq:upbd_sandhanala} is Lemma $10$ in \cite{sadhanala16}; it bounds the variance error by bounding $\sum_{i=1}^n \frac{1}{(1+\reg \lambda_i)^2}$. This latter bound is tight as the analysis steps obviously make explicit use of the  expressions for the eigenvalues of $L$. It is possible that the steps involved in \cite[Lemma 10]{sadhanala16} could be appropriately used to further tighten our bound in Corollary \ref{cor:ucqp_den_func_mod} when $G = P_n$. However the main purpose of our analysis is to work with general connected graphs $G$, and to derive general error bounds which depend on the spectrum of the Laplacian of $G$. It is then not surprising that instantiating these general bounds to particular graphs yields sub-optimal error bounds.
\end{remark}
Kirichenko and van Zanten \cite{kirichenko2017} considered a Bayesian regularization framework for estimating $\theta^\star$. 
They make an asymptotic shape assumption on the graph, namely that $G$ looks like a $r$-dimensional grid with $n$ vertices as $n \rightarrow \infty$. 
More precisely, they assume that $\lambda_{n-i} = \Theta((i/n)^{2/r})$ for $1 \leq i \leq \kappa n$ for some $\kappa \in (0,1)$. The smoothness of $\theta^\star$ 
is captured by the assumption $(\theta^\star)^\top (I + (n^{2/r} L)^{\beta}) \theta^\star \leq C$ where $\beta, C > 0$. Under these assumptions, with an appropriate assumption on the prior distribution for a randomly generated $\theta \in \matR^n$, it is shown \cite[Theorem 3.2]{kirichenko2017} that for $n$ large enough,  
$$\frac{1}{n}\|\theta - \theta^\star\|_2^2 = O(n^{-\frac{2\beta}{2\beta + r}})$$ 
holds with high probability. This is then used to show \cite[Theorem 5.1]{kirichenko2017} that the posterior contracts around $\theta^\star$ at the rate 
$n^{-\frac{\beta}{2\beta + r}}$. This result was later shown to be optimal by Kirichenko and van Zanten \cite{kirichenko2018} under the same set of smoothness and asymptotic shape assumptions on $\theta^\star$ and $G$ respectively.
%
%
\subsection{The analysis technique of Cucuringu and Tyagi \cite{CMT18_long}} It is important to understand the general idea behind the analysis technique in \cite{CMT18_long} for \eqref{prog:trs} that leads to the estimation error bound in \eqref{eq:cmt_err_bd}. Denoting $F(g) = \norm{g - z}_2^2 + \reg g^* L g$ to be the objective function, the main observation is that by feasibility of the ground truth $h \in \calC_n$, we have $F(\gest) \leq F(h)$ for any solution $\gest$. Then after rearranging the terms followed by some simplification, one can readily check that the above inequality is equivalent to
\begin{equation} \label{eq:cmt_inter_ineq}
\norm{\gest - h}_2^2 \leq \norm{z - h}_2^2 - \reg \gest^* L \gest + 2\real(\gest^* (z-h)) + \reg h^* L h.    
\end{equation}
Now if $z$ is generated randomly as in \eqref{eq:noise_mod}, we know that $\norm{z-h}_2^2 \lesssim \sigma^2 n$ w.h.p if $\sigma \lesssim 1$. Moreover, the term $\real(\gest^* (z-h))$ is bounded via Cauchy-Schwartz to obtain $\real(\gest^* (z-h)) \leq \sqrt{n} \norm{z-h}_2 \lesssim \sigma n$ w.h.p. Plugging these bounds in \eqref{eq:cmt_inter_ineq} leads to the bound
\begin{equation} \label{eq:cmt_inter_ineq_1}
\norm{\gest - h}_2^2 \lesssim \sigma n - \reg \gest^* L \gest + \reg h^* L h.    
\end{equation}
The final step in the analysis requires lower bounding the quadratic term $\gest^* L \gest$, which first involves utilising the expression of the \eqref{prog:trs} solution $\gest$ to show that \cite[Lemma 5]{CMT18_long} 
\begin{equation} \label{eq:cmt_inter_ineq_2}
    \gest^* L \gest \gtrsim \frac{1}{(1+2\reg \Delta)^2} z^* L z
\end{equation}
and subsequently using concentration inequalities to show that \cite[Proposition 2]{CMT18_long} w.h.p., 
$z^* L z \gtrsim \reg \Delta \sigma^4 + h^* L h$ when $\sigma \lesssim 1$. Plugging these considerations in \eqref{eq:cmt_inter_ineq_1}, and noting that $h^* L h \lesssim \frac{M^2 \Delta^3}{n}$ finally leads to the bound
\begin{equation*} 
    \norm{\gest - h}_2^2 \lesssim \sigma n + \reg \frac{M^2 \Delta^3}{n} - \frac{\reg^2 \Delta}{(1+2\reg \Delta)^2} \sigma^4. 
\end{equation*}
Using Proposition \ref{prop:entry_proj}, we obtain the bound stated in \eqref{eq:cmt_err_bd} since $\frac{\reg^2 \Delta}{(1+2\reg \Delta)^2}$ is always less than $\frac{1}{4\Delta}$. 

We believe that certain steps in the above analysis can likely be improved. For instance the bound on $\real(\gest^* (z-h))$ could be perhaps improved by using the expression for the \eqref{prog:trs} solution $\gest$, along with concentration inequalities. Furthermore, the lower bound in \eqref{eq:cmt_inter_ineq_2} is almost certainly sub-optimal as can be seen from the proof of \cite[Lemma 5]{CMT18_long}. But it seems unlikely that the ensuing improvements will improve the bounds drastically, and would probably at best improve the term $\sigma n$ to $\sigma^2 n$.   
%
\subsection{Future work} \label{subsec:fut_work}
There are several important directions for future work.
\begin{enumerate}
    \item \textit{Optimality of the error bounds.} The $\ell_2$ error bounds that we derived for the \eqref{prog:trs} and \eqref{prog:ucqp} estimators are certainly not optimal. For instance as noted earlier in Remark \ref{rem:ucqp_lw_bd_subopt}, when $G = P_n$ and $\smooth_n \asymp 1/n$, the $\ell_2$ error bound for \eqref{prog:ucqp} is $O(n^{1/3})$ while the optimal scaling should be $O(n^{1/6})$.  Admittedly, our analysis does involve certain simplifications at different steps in order to make the calculations more tractable - especially in the derivation of the choice of the regularization parameter $\reg$. In particular, since we work with general connected graphs and do not make any assumption on the spectrum of the Laplacian, hence the central theme behind our analysis -- which is to separate the low and high frequency spectrum of the Laplacian -- can be considered a bit crude for certain graphs whose Laplacian exhibits a more continuous spectrum (such as the path graph). In general, if one restricts the analysis to special families of graphs with the spectrum of the Laplacian's possessing a specific structure (e.g., $G = P_n$; recall Remark \ref{rem:sadhanala_bounds}), it is plausible that a more refined analysis could be carried out with better error rates. Otherwise, providing an ``optimal'' analysis in its full generality would involve finding the best choice of $\reg$ which -- as evidenced by the proof of Lemma \ref{lem:ucqp_lem_err_1} -- seems challenging. Nevertheless, this is an interesting question to consider for future work.  
    
    \item \textit{$\ell_{\infty}$ estimation error rates.} While our focus throughout has been on deriving  $\ell_2$ error bounds, an interesting but more challenging task would be to derive $\ell_{\infty}$ error bounds for the \eqref{prog:trs} and \eqref{prog:ucqp} estimators. Such a result would be especially useful for the problem of unwrapping noisy modulo samples of a function $f$ since it would enable us to obtain uniform error rates for the unwrapped samples of $f$, using the aforementioned results of Fanu\"el and Tyagi \cite{fanuel20}.
\end{enumerate}

\section*{Acknowledgements}
I would like to thank St\'{e}phane Chr\'{e}tien and Micha\"el Fanuel for carefully reading a preliminary version of the draft, and for providing useful feedback; Alain Celisse for the very helpful technical discussions during the early stages of this work.

\bibliographystyle{plain}
\bibliography{references}

\begin{thebibliography}{10}

\bibitem{Naka17}
S.~Adachi, S.~Iwata, Y.~Nakatsukasa, and A.~Takeda.
\newblock Solving the trust-region subproblem by a generalized eigenvalue
  problem.
\newblock {\em SIAM Journal on Optimization}, 27(1):269--291, 2017.

\bibitem{belkin04}
M.~Belkin, I.~Matveeva, and P.~Niyogi.
\newblock Regularization and semi-supervised learning on large graphs.
\newblock In {\em Learning Theory}, pages 624--638, 2004.

\bibitem{bellec2019}
P.C. Bellec.
\newblock Concentration of quadratic forms under a bernstein moment assumption,
  2019.

\bibitem{bhandari17}
A.~Bhandari, F.~Krahmer, and R.~Raskar.
\newblock On unlimited sampling.
\newblock ArXiv e-prints, arXiv:1707.06340v1, 2017.

\bibitem{sparse_unlim18}
A.~{Bhandari}, F.~{Krahmer}, and R.~{Raskar}.
\newblock Unlimited sampling of sparse signals.
\newblock In {\em 2018 IEEE International Conference on Acoustics, Speech and
  Signal Processing (ICASSP)}, pages 4569--4573, 2018.

\bibitem{sparse_unlimsin18}
A.~{Bhandari}, F.~{Krahmer}, and R.~{Raskar}.
\newblock Unlimited sampling of sparse sinusoidal mixtures.
\newblock In {\em 2018 IEEE International Symposium on Information Theory
  (ISIT)}, pages 336--340, 2018.

\bibitem{conc_book}
St{\'e}phane Boucheron, Gabor Lugosi, and Pascal Massart.
\newblock {\em {Concentration inequalities : a non asymptotic theory of
  independence}}.
\newblock {Oxford University Press}, 2013.

\bibitem{brouwer12}
A.E. Brouwer and W.H. Haemers.
\newblock {\em Spectra of Graphs}.
\newblock New York, NY, 2012.

\bibitem{Chavez02}
S.~Chavez, Q.S Xiang, and L.~An.
\newblock Understanding phase maps in mri: a new cutline phase unwrapping
  method.
\newblock {\em IEEE Transactions on Medical Imaging}, 21(8):966--977, 2002.

\bibitem{cmt_aistats18}
M.~Cucuringu and H.~Tyagi.
\newblock On denoising modulo 1 samples of a function.
\newblock In {\em Proceedings of the Twenty-First International Conference on
  Artificial Intelligence and Statistics}, volume~84, pages 1868--1876, 2018.

\bibitem{CMT18_long}
M.~Cucuringu and H.~Tyagi.
\newblock Provably robust estimation of modulo 1 samples of a smooth function
  with applications to phase unwrapping.
\newblock {\em Journal of Machine Learning Research}, 21(32):1--77, 2020.

\bibitem{fanuel20}
M.~Fanuel and H.~Tyagi.
\newblock Denoising modulo samples: k-{NN} regression and tightness of {SDP}
  relaxation.
\newblock In preparation, 2020.

\bibitem{graham_insar}
L.~C. Graham.
\newblock Synthetic interferometer radar for topographic mapping.
\newblock {\em Proceedings of the IEEE}, 62(6):763--768, 1974.

\bibitem{Hager01}
W.W. Hager.
\newblock Minimizing a quadratic over a sphere.
\newblock {\em SIAM J. on Optimization}, 12(1):188--208, 2001.

\bibitem{hedley92}
M.~Hedley and D.~Rosenfeld.
\newblock A new two-dimensional phase unwrapping algorithm for mri images.
\newblock {\em Magnetic Resonance in Medicine}, 24(1):177--181, 1992.

\bibitem{huang2012path}
H.Y.H Huang, L.~Tian, Z.~Zhang, Y.~Liu, Z.~Chen, and G.~Barbastathis.
\newblock Path-independent phase unwrapping using phase gradient and
  total-variation (tv) denoising.
\newblock {\em Optics express}, 20(13):14075--14089, 2012.

\bibitem{Kester}
W.~Kester.
\newblock Mt-025 tutorial adc architectures vi: Folding adcs.
\newblock Analog Devices, Tech. report, 2009.

\bibitem{kirichenko2017}
A.~Kirichenko and H.~van Zanten.
\newblock Estimating a smooth function on a large graph by bayesian laplacian
  regularisation.
\newblock {\em Electron. J. Statist.}, 11(1):891--915, 2017.

\bibitem{kirichenko2018}
A.~Kirichenko and H.~van Zanten.
\newblock Minimax lower bounds for function estimation on graphs.
\newblock {\em Electron. J. Statist.}, 12(1):651--666, 2018.

\bibitem{laut73}
P.~Lauterbur.
\newblock Image formation by induced local interactions: examples employing
  nuclear magnetic resonance.
\newblock {\em Nature}, 242:190--191, 1973.

\bibitem{Liu2017}
H.~Liu, Man-Chung Yue, and A.~Man-Cho~So.
\newblock On the estimation performance and convergence rate of the generalized
  power method for phase synchronization.
\newblock {\em SIAM Journal on Optimization}, 27(4):2426--2446, 2017.

\bibitem{MusaJG18}
O.~Musa, P.~Jung, and N.~Goertz.
\newblock Generalized approximate message passing for unlimited sampling of
  sparse signals.
\newblock In {\em 2018 {IEEE} Global Conference on Signal and Information
  Processing, GlobalSIP}, pages 336--340, 2018.

\bibitem{nemirovski2000topics}
A.~Nemirovski.
\newblock Topics in non-parametric statistics.
\newblock {\em Ecole d’Et{\'e} de Probabilit{\'e}s de Saint-Flour}, 28:85,
  2000.

\bibitem{Prati90}
C.~Prati, M.~Giani, and N.~Leuratti.
\newblock Sar interferometry: A 2-d phase unwrapping technique based on phase
  and absolute values informations.
\newblock In {\em 10th Annual International Symposium on Geoscience and Remote
  Sensing}, pages 2043--2046, 1990.

\bibitem{pratt88}
R.G. Pratt and M.H. Worthington.
\newblock The application of diffraction tomography to cross‐hole seismic
  data.
\newblock {\em GEOPHYSICS}, 53(10):1284--1294, 1988.

\bibitem{RHEE03}
J.~Rhee and Y.~Joo.
\newblock Wide dynamic range cmos image sensor with pixel level adc.
\newblock {\em Electronics Letters}, 39(4):360--361, 2003.

\bibitem{Rivera04half}
M.~Rivera and J.L. Marroquin.
\newblock Half-quadratic cost functions for phase unwrapping.
\newblock {\em Opt. Lett.}, 29(5):504--506, 2004.

\bibitem{rudresh_wavelet}
S.~{Rudresh}, A.~{Adiga}, B.~A. {Shenoy}, and C.~S. {Seelamantula}.
\newblock Wavelet-based reconstruction for unlimited sampling.
\newblock In {\em 2018 IEEE International Conference on Acoustics, Speech and
  Signal Processing (ICASSP)}, pages 4584--4588, 2018.

\bibitem{sadhanala16}
V.~Sadhanala, Yu-Xiang Wang, and R.J. Tibshirani.
\newblock Total variation classes beyond 1d: Minimax rates, and the limitations
  of linear smoothers.
\newblock In {\em Proceedings of the 30th International Conference on Neural
  Information Processing Systems}, NIPS'16, page 3521–3529, 2016.

\bibitem{shah2018signal}
Viraj Shah and Chinmay Hegde.
\newblock Signal reconstruction from modulo observations, 2018.

\bibitem{shuman13}
D.~I. {Shuman}, S.~K. {Narang}, P.~{Frossard}, A.~{Ortega}, and
  P.~{Vandergheynst}.
\newblock The emerging field of signal processing on graphs: Extending
  high-dimensional data analysis to networks and other irregular domains.
\newblock {\em IEEE Signal Processing Magazine}, 30(3):83--98, 2013.

\bibitem{Sorensen82}
D.C. Sorensen.
\newblock Newton’s method with a model trust region modification.
\newblock {\em SIAM Journal on Numerical Analysis}, 19(2):409--426, 1982.

\bibitem{takeda96}
M.~Takeda and T.~Abe.
\newblock Phase unwrapping by a maximum cross-amplitude spanning tree
  algorithm: a comparative study.
\newblock {\em Optical Engineering}, 35:35 -- 35 -- 7, 1996.

\bibitem{Towers91}
D.~P. {Towers}, T.~R. {Judge}, and P.~J. {Bryanston-Cross}.
\newblock {Automatic interferogram analysis techniques applied to
  quasi-heterodyne holography and ESPI}.
\newblock {\em Optics and Lasers in Engineering}, 14:239--281, 1991.

\bibitem{yamaguchi16}
T.~Yamaguchi, H.~Takehara, Y.~Sunaga, M.~Haruta, M.~Motoyama, Y.~Ohta, T.~Noda,
  K.~Sasagawa, T.~Tokuda, and J.~Ohta.
\newblock Implantable self-reset cmos image sensor and its application to
  hemodynamic response detection in living mouse brain.
\newblock {\em Japanese Journal of Applied Physics}, 55(4S):04EM02, 2016.

\bibitem{zebker86}
H.A. Zebker and R.M. Goldstein.
\newblock Topographic mapping from interferometric synthetic aperture radar
  observations.
\newblock {\em Journal of Geophysical Research: Solid Earth},
  91(B5):4993--4999, 1986.

\end{thebibliography}

\newpage
\appendix

\section{\secrev{Summary of notation}} \label{appsec:summary_notation}

\begin{center}
\begin{table}[!ht]
\begin{tabular}{ | m{4cm}| m{10cm} | } 
\hline
\textit{Symbol} & \textit{Definition} \\
\hline\hline
 $G = ([n], E)$ & Undirected, connected graph with $n$ vertices and edge set $E$  \\ 
\hline
 $\triangle > 0$ & Maximum degree of $G$ \\ 
\hline
$L \in \matR^{n \times n}$ & Laplacian matrix  of $G$ \\ 
\hline
$\mathcal{N}(L)$ & Null space of $L$ which is span$\set{q_n}$ since $G$ is connected \\
\hline
$\lambda_n = 0 < \lambda_{n-1} = \lambda_{\min} \leq \lambda_{n-2} \leq \cdots \leq \lambda_1$ & Eigenvalues of $L$ \\
\hline
$q_i \in \matR^n$; $i=1,\dots,n$ & Corresponding eigenvectors of $L$ \\
\hline
    $\calL_{\lambda} \subset [n-1]$ & $\calL_{\lambda} := \set{j \in [n-1]: \lambda_j < \lambda}$ (defined for $\lambda \in [\lambmin, \lambda_1]$, see \eqref{eq:low_freq_set}) \\
\hline
$\calC_n \subset \mathbb{C}^n$ & $\calC_n := \set{u \in \mathbb{C}^n: \abs{u_i} = 1; \ i=1,\dots,n}$ \\
\hline
$h \in \calC_n$ & Ground-truth signal \\
\hline
$\smooth_n \geq 0$ & Smoothness parameter, $h^* L h \leq \smooth_n$ \\
\hline
$\sigma \geq 0$ & Noise level \\
\hline
$z \in \calC_n$ & Noisy measurements of $h$ (see \eqref{eq:noise_mod}) \\
\hline
$\reg \geq 0$ & Smoothness regularization parameter in \eqref{prog:ucqp} and \eqref{prog:trs} \\
\hline
$G = K_n$ & $E = \set{\set{i,j}: i \neq j \in [n]}$ ($G$ is a complete graph) \\
\hline
$G = P_n$ & $E = \set{\set{i,i+1}: i=1,\dots,n}$ ($G$ is a path graph) \\
\hline
$G = S_n$ & $E = \set{\set{i,i_0}: i \neq i_0 \in [n]}$ for a given $i_0 \in [n]$ ($G$ is a star graph) \\
\hline
$M > 0$ & Lipschitz constant of $f:[0,1] \rightarrow \matR$ in the example in Section \ref{subsec:prob_setup} \\
\hline
\end{tabular}
\caption{Summary of symbols used throughout the paper along with their definitions.}
\end{table}
\end{center}

%
%
\section{Proof of Proposition \ref{prop:noise_expec_iden}} \label{appsec:proof_prop_noise_expec}
\begin{enumerate}
\item This follows from the fact $\expec[e^{\iota 2\pi \eta}] = e^{-2\pi\sigma^2}$ for $\eta \sim \calN(0,\sigma^2)$.

\item We have
\begin{equation*}
  \expec\left[\abs{\dotprod{z- e^{-2\pi^2 \sigma^2} h}{u}}^2 \right] = \sum_{i=1}^{n} u_i^2 \expec \left[\abs{z_i - e^{-2\pi^2 \sigma^2} h_i}^2 \right] = \sum_{i=1}^n u_i^2 (1 + e^{-4\pi^2\sigma^2} - 2e^{-2\pi^2\sigma^2} \underbrace{\expec[\real(z_i^* h_i)]}_{= e^{-2\pi^2\sigma^2}}).
\end{equation*}

\item This follows from Part $2$ by noting that 
\begin{equation*}
    \expec\left[\abs{\dotprod{z}{u}}^2 \right] = e^{-4 \pi^2\sigma^2}\abs{\dotprod{h}{u}}^2 + \expec \left[\abs{\dotprod{z - e^{-2\pi^2\sigma^2} h}{u}}^2 \right].
\end{equation*}

\item Use Part $2$ and the fact that for any orthonormal basis $\set{u_j}_{j=1}^n$ of $\matR^n$, $$\norm{z - \expn{2} h}_2^2 = \sum_{j=1}^n \abs{\dotprod{z- e^{-2\pi^2 \sigma^2} h}{u_j}}^2. $$

\item $\expec[\norm{z-h}_2^2] = 2n - 2 \sum_{i=1}^n \real(\expec[z_i^* h_i]) = 2n (1-e^{-2\pi^2 \sigma^2})$ where the last identity uses Part $1$.
\end{enumerate}
The bounds in \eqref{eq:expec_zh_dist_bds} follow from the following standard fact. For $x \geq 0$, we have that $x - \frac{x^2}{2} \leq 1 - e^{-x} \leq x$. Hence, if $x \in [0,1]$, this implies that $\frac{x}{2} \leq 1-e^{-x} \leq x$.

\section{Proof of Proposition \ref{prop:conc_bounds}} \label{appsec:proof_prop_conc}
Before the proof, we recall some concentration inequalities that we will use. The first of these is the standard Bernstein inequality.
\begin{theorem}[{\cite[Corollary 2.11]{conc_book}}] \label{thm:bern_ineq}
Let $X_1,\dots,X_n$ be independent random variables with $\abs{X_i} \leq b$ for all $i$, 
and $v = \sum_{i=1}^n \expec[X_i^2]$. Then for any $t \geq 0$,
\begin{equation*}
	\prob \left(\sum_{i=1}^n (X_i - \expec[X_i]) \geq t \right) \leq \exp\left(-\frac{t^2}{2(v + \frac{bt}{3})}\right).
\end{equation*}
\end{theorem}
Note that replacing $X_i$ with $-X_i$ gives us the lower tail estimate. Next, we will use a recent, sharper version of the Hanson-Wright inequality due to Bellec \cite{bellec2019}, for concentration of random quadratic forms. We state (for our purposes) a shorter version of this theorem.
\begin{theorem}[{\cite[Theorem 3]{bellec2019}}] \label{thm:bellec}
Let $\vecxi := (\xi_1,\dots,\xi_n)^T$ be centered, independent (real-valued) random variables, with $\nu_i^2 = \expec[\xi_i^2]$, satisfying for some $K > 0$ the Bernstein condition
\begin{equation*}
    \forall p \geq 1: \quad \expec \abs{\xi_i}^{2p} \leq \frac{p!}{2} \nu_i^2 K^{2(p-1)}.
\end{equation*}
For any real  matrix $A \in \matR^{n \times n}$, and any $x > 0$, we have with probability at least $1 - e^{-x}$ that 
\begin{equation*}
    \vecxi^T A \vecxi - \expec[\vecxi^T A \vecxi] \leq 256 K^2 \norm{A}_2 x + 8\sqrt{3} K \norm{A D_{\nu}}_F \sqrt{x}
\end{equation*}
where $D_{\nu}:= \diag(\nu_1,\dots,\nu_n)$.
\end{theorem}
The Bernstein condition is satisfied, for example, by centered  bounded random variables almost surely bounded by $K$, which will be the case in our setting. Note that replacing $A$ with $-A$ in Theorem \ref{thm:bellec} gives us the lower tail estimate.
\begin{proof}[Proof of Proposition \ref{prop:conc_bounds}]
We will denote $\zbar = \expec[z] \ (= \expn{2} h)$ and also denote $v = v_R + \iota v_I$ for any $v \in \mathbb{C}^n$. 
\paragraph{(1) Proof of \ref{prop:conc_bds_1}.}
Now note that 
\begin{align} \label{eq:part1_tmp1}
z^* UU^T z = (z-\zbar)^* UU^T (z - \zbar) + 2\real((z-\zbar)^* UU^T \zbar) + \expn{4} h^* UU^T h    
\end{align}
and so we will focus on lower bounding the first two terms on the RHS. In particular, we will bound the first term using Theorem \ref{thm:bellec} (with $A = -UU^T$) and the second term using Theorem \ref{thm:bern_ineq}.

\noindent \textit{(1a) Bounding the first term in \eqref{eq:part1_tmp1}.}
Since 
\begin{equation} \label{eq:quad_split_1}
  (z-\zbar)^* UU^T (z - \zbar) = (z_R - \zbar_R)^* UU^T (z_R - \zbar_R) +  (z_I - \zbar_I)^* UU^T (z_I - \zbar_I) 
\end{equation}
therefore we will bound each of the two terms in \eqref{eq:quad_split_1} using Theorem \ref{thm:bellec} with $A = -UU^T$. In fact, we will do this only for the first term since the bound on the other term follows in an analogous manner. To this end, note that 
$\expec[z_{R,i}] = \expn{2} h_{R,i}$ and also
\begin{equation*}
\expec[z_{R,i}^2] = h_{R,i}^2 \left(\frac{1 + \expn{8}}{2}\right)    + h_{I,i}^2 \left(\frac{1 - \expn{8}}{2}\right).
\end{equation*}
Then a simple calculation reveals the bound
\begin{equation} \label{eq:nuRmax_bd}
 \nu_{R,i}^2 :=  \expec[z_{R,i}^2] - (\expec[z_{R,i}])^2 = \frac{h_{R,i}^2}{2} (1 - \expn{4})^2 + \frac{h_{I,i}^2}{2} (1-\expn{8}) \leq 1-\expn{8}.
\end{equation}
Denoting $D_{\nu,R} := \diag(\nu_{R,1},\dots,\nu_{R,n})$, observe that 
\begin{equation*}
\norm{UU^T D_{\nu,R}}_F = \norm{Q^T D_{\nu,R}}_F \leq \sqrt{k} \norm{D_{\nu,R}}_2 \leq \sqrt{k} (1-\expn{8})
\end{equation*}
where the last inequality uses \eqref{eq:nuRmax_bd}. Since $\abs{z_{R,i} - \zbar_{R,i}} \leq 2$ for each $i$, we obtain from Theorem \ref{thm:bellec} that with probability at least $1-e^{-x}$, 
\begin{equation*}
 (z_R - \zbar_R)^* UU^T (z_R - \zbar_R) \geq \expec[(z_R - \zbar_R)^* UU^T (z_R - \zbar_R)] - (1024 x + 16\sqrt{3k}  (1-\expn{8}) \sqrt{x}).  
\end{equation*}
The same analysis holds for the other term in \eqref{eq:quad_split_1}. Hence plugging these estimates in \eqref{eq:quad_split_1}, using Proposition \ref{prop:noise_expec_iden}, and setting $x = 2\log n$, we obtain with probability at least $1- \frac{2}{n^2}$ that 
\begin{equation} \label{eq:bd_first_term}
   (z-\zbar)^* UU^T (z - \zbar) \geq k(1 - \expn{4}) - 4096 \log n - 32\sqrt{6k}  (1 - \expn{8}) \sqrt{\log n}.    
\end{equation}

\noindent \textit{(1b) Bounding the second term in \eqref{eq:part1_tmp1}.}
We can write
\begin{equation} \label{eq:quad_split_2}
    \real((z-\zbar)^* UU^T \zbar) = (z_{R} - \zbar_R)^T  Q Q^T \zbar_R + (z_{I} - \zbar_I)^T  Q Q^T \zbar_I 
\end{equation}
so we will bound each of the two terms in \eqref{eq:quad_split_2} by Bernstein inequality. In particular we will only do this for the first term since the other term can be bounded analogously. To this end, denoting $u_R = QQ^T \zbar_R$, we have 
\begin{align*}
    (z_{R} - \zbar_R)^T  Q Q^T \zbar_R = \sum_{i=1}^n \underbrace{(z_{R,i} - \zbar_{R,i}) u_{R,i}}_{X_{R,i}}
\end{align*}
which is the sum of zero mean independent random variables. We can bound $\abs{X_{R,i}}$ uniformly as
\begin{align*}
    \abs{X_{R,i}} \leq 2 \norm{u_R}_{\infty} \leq \underbrace{2\expn{2} \norm{QQ^T h_R}_{\infty}}_{b_{R,\max}}; \quad i=1,\dots,n,
\end{align*}
and the variance term
\begin{equation*}
 \sum_{i=1}^n \expec[X_{R,i}^2] \leq (1-\expn{8}) \norm{u_R}_2^2 = (1-\expn{8}) \expn{4} \norm{QQ^T h_R}_2^2 = v_{R,\max}.
\end{equation*}
Now using Theorem \ref{thm:bern_ineq} with $v = v_{\max,R}$, $b = b_{\max,R}$, and $t = -\frac{2}{3} \log n(b_{\max,R} + \sqrt{b_{\max,R}^2 + 9 v_{\max,R}})$, we obtain
\begin{equation} \label{eq:real_part_bern_1}
 \prob(-(z_{R} - \zbar_R)^T  Q Q^T \zbar_R \geq \frac{2}{3} \log n(b_{\max,R} + \sqrt{b_{\max,R}^2 + 9 v_{\max,R}})) \leq \frac{1}{n^2}.
\end{equation}
Similarly, one can show that 
\begin{equation} \label{eq:imag_part_bern_1}
 \prob \left(-(z_{I} - \zbar_I)^T  Q Q^T \zbar_I \geq \frac{2}{3} \log n(b_{\max,I} + \sqrt{b_{\max,I}^2 + 9 v_{\max,I}}) \right) \leq \frac{1}{n^2}.
\end{equation}
with $b_{\max,I} = 2\expn{2} \norm{QQ^T h_I}_{\infty}$, $v_{\max,I} = (1-\expn{8}) \expn{4} \norm{QQ^T h_I}_2^2$.
Therefore combining \eqref{eq:real_part_bern_1}, \eqref{eq:imag_part_bern_1} in \eqref{eq:quad_split_2}, we have w.p at least $1 - \frac{2}{n^2}$ that 
\begin{align}
    &\real((z-\zbar)^* UU^T \zbar) \nonumber \\
    &\geq -\frac{2}{3} \log n \left(b_{\max,R} + b_{\max,I} + \sqrt{b_{\max,R}^2 + 9 v_{\max,R}} + \sqrt{b_{\max,I}^2 + 9 v_{\max,I}} \right) \nonumber \\
    &\geq -\frac{2}{3} \log n \left(8 \expn{2} \norm{QQ^T h}_{\infty} + 3(\sqrt{v_{\max,R}} + \sqrt{v_{\max,I}}) \right) \nonumber \\
    &\geq -\frac{2}{3} \log n \left(8 \expn{2} \norm{QQ^T h}_{\infty} + 6 \sqrt{(1-\expn{8}) \expn{4}} \norm{Q Q^T h}_2 \right) \nonumber \\
    &\geq -\frac{16}{3} \expn{2} \log n  \left(\norm{QQ^T h}_{\infty} +  \sqrt{(1-\expn{8})} \norm{Q Q^T h}_2 \right). \label{eq:sec_term_bd}
\end{align}
Hence plugging \eqref{eq:bd_first_term} and \eqref{eq:sec_term_bd} in \eqref{eq:part1_tmp1} and applying the union bound, we obtain the statement of part \ref{prop:conc_bds_1} of the proposition after a slight simplification involving the constants.

\paragraph{(2) Proof of \ref{prop:conc_bds_2}.} This follows in an identical manner as \eqref{eq:bd_first_term} by using Theorem \ref{thm:bellec} with $A = UU^T$.

\paragraph{(3) Proof of \ref{prop:conc_bds_3}.} Observe that 
\begin{align} \label{eq:z_h_norm_bd}
\norm{z-h}_2^2 = 2n - 2\real(z^* h) = 2n (1 - \expn{2}) - 2 ((z_R - \zbar_R)^T h_R + (z_I - \zbar_I)^T h_I).    
\end{align}
We will bound $(z_R - \zbar_R)^T h_R$ from above via Theorem \ref{thm:bern_ineq}; the same bound will hold for the term $(z_I - \zbar_I)^T h_I$ which then yields the stated bound in the proposition. 
To this end, note that 
\begin{align*}
    (z_{R} - \zbar_R)^T  h_R = \sum_{i=1}^n \underbrace{(z_{R,i} - \zbar_{R,i}) h_{R,i}}_{X_{R,i}}
\end{align*}
which is the sum of zero mean independent random variables. We can bound $\abs{X_{R,i}}$ uniformly as $\abs{X_{R,i}} \leq 2$ for each $i$, and the variance term
\begin{equation*}
 \sum_{i=1}^n \expec[X_{R,i}^2] \leq (1-\expn{8}) \norm{h_R}_2^2 \leq (1-\expn{8}) n =: v_{\max}.
\end{equation*}
Then applying Theorem \ref{thm:bern_ineq} with $v = v_{\max}$ and $b = 2$ and $t = \frac{2}{3} \log n (2 + \sqrt{4 + 9 v_{\max}})$ yields
\begin{equation} \label{eq:zreal_dprod_bd}
\prob\left((z_{R} - \zbar_R)^T  h_R \leq -\frac{2}{3} \log n \left(2 + \sqrt{4 + 9n (1-\expn{8} ) } \right)\right) \leq \frac{1}{n^2}.
\end{equation}
The same bound holds for the term $(z_I - \zbar_I)^T h_I$ as well, hence plugging these bounds in \eqref{eq:z_h_norm_bd}, together with the union bound on the success probability, yields the statement of part \ref{prop:conc_bds_3}.

\paragraph{(4) Proof of \ref{prop:conc_bds_4}.} The proof is along the lines of that for part \ref{prop:conc_bds_3}. Observe that 
\begin{align}
    \norm{z - \expn{2} h}_2^2 &= n + \expn{4} n - 2\real(z^* h) \expn{2} \nonumber \\
    &= n + \expn{4} n - 2\real((z-\zbar)^* h) \expn{2} - 2\real(\zbar^* h) \expn{2} \nonumber \\
    &= n (1 - \expn{4}) - 2\expn{2} ((z_R - \zbar_R)^T h_R + (z_I - \zbar_I)^T h_I). \label{eq:z_cen_h_norm_bd}
\end{align}
Then bounding the terms $(z_{R} - \zbar_R)^T, (z_{I} - \zbar_I)^T$ as in \eqref{eq:zreal_dprod_bd}, and plugging these bounds in \eqref{eq:z_cen_h_norm_bd}, we obtain the statement of part \ref{prop:conc_bds_4} after the simplification $\expn{2} \leq 1$.
\end{proof}

%
%
\section{Proof of Corollary \ref{cor:Pn_res_bet_conds}} \label{appsec:proof_Pn_res_bet_conds}
Recall from Corollary \ref{cor:ucqp_den_expec} that $\lambda_j = 4 \sin^2[\frac{\pi}{2n} (n-j)]$ for $j=1,\dots,n$. Hence for $k=1,\dots,n-1$, 
\begin{equation*}
    \lambda_{n-k} = 4 \sin^2 \left(\frac{\pi}{2n} k \right) \asymp \frac{k^2}{n^2} \quad \text{ and } \quad \lambda_{n-k+1} = 4 \sin^2 \left(\frac{\pi}{2n} (k-1) \right) \asymp \frac{(k-1)^2}{n^2}
\end{equation*}
where we see that $\lambda_{n-k+1} < \lambda_{n-k}$ for each $k$. 
In particular, $\lambda_1 \asymp 1$ and  $\lambmin \asymp \frac{1}{n^2}.$ Also recall that $\lambbar \asymp \frac{1}{n^{2(1-\theta)}}$ for $\theta \in [0,1)$ which implies $\abs{\calL_{\lambbar}} \lesssim n^{\theta}$.

Plugging the above bounds in the error bound in Theorem \ref{thm:trs_err_bd_prob}, we obtain
    \begin{equation} \label{eq:app_trs_Pn_temp1}
       \norm{\frac{\gest}{\abs{\gest}} - h}_2^2 \lesssim \sigma n^{2(1-\theta)} \left(M + \frac{(k-1)^4}{n^2 M} \right) + \sigma^2 (n^{\theta} + \sqrt{n^{\theta} \log n}) + \sigma^4 n + \log n + \frac{M^4 n}{k^4}.
    \end{equation}
Setting $k = \lfloor n^{1/2} \rfloor$ in  \eqref{eq:app_trs_Pn_temp1} simplifies the bound to
    \begin{equation} \label{eq:app_trs_Pn_temp2}
       \norm{\frac{\gest}{\abs{\gest}} - h}_2^2 \lesssim \sigma n^{2(1-\theta)} \left(M + \frac{1}{M} \right) + \sigma^2 (n^{\theta} + \sqrt{n^{\theta} \log n}) + \sigma^4 n + \log n + \frac{M^4}{n}.
    \end{equation}
note that the choice $\theta = 2/3$ ``balances'' the exponents of $n$ in the first two terms in \eqref{eq:app_trs_Pn_temp2}. For this choice of $\theta$, we can see that \eqref{eq:app_trs_Pn_temp2} simplifies to the stated error bound in the first part of the Corollary when $n \gtrsim 1$.

For $k = \lfloor n^{1/2} \rfloor$ and $\theta = 2/3$, we have $\lambda_{n-k}$, $\lambda_{n-k+1}  \asymp \frac{1}{n}$, $\lambbar = \frac{1}{n^{2/3}}$ and $\abs{\calL_{\lambbar}} \lesssim n^{2/3}$. Then, 
\begin{equation*}
 \min\set{n \lambda_{n-k}, n \lambbar} \asymp \min \set{1, n^{1/3}} = 1   
\end{equation*}
and since $\smooth_n \asymp M^2/n$, therefore $\smooth_n \lesssim \min\set{n \lambda_{n-k}, n \lambbar}$ is ensured if $n \gtrsim M^2$. The condition on $\sigma$ in the first part follows readily by applying the above bounds to the conditions on $\sigma$ in Theorem \ref{thm:trs_err_bd_prob}. This completes the proof of the first part of Corollary \ref{cor:Pn_res_bet_conds}.  

The statement of the second part follows in a straightforward manner upon applying the above considerations to Theorem \ref{thm:trs_prob_denoise}.
We only remark that if $n \gtrsim 1$, then the condition $1 + \abs{\calL_{\lambbar}} + \sqrt{(1 + \abs{\calL_{\lambbar}}) \log n} \lesssim \varepsilon n$ is ensured provided $n^{2/3} \lesssim \varepsilon n$ or equivalently $n \gtrsim (1/\varepsilon)^3$ (this subsumes the requirement $n \gtrsim 1$).

\end{document}